\documentclass[reqno,10pt,centertags]{amsart}
\usepackage[letterpaper,margin=1.5in,bottom=1.3in]{geometry}
\usepackage{amsmath,amsthm,amscd,amssymb,latexsym,enumerate}

\usepackage{hyperref}

\newcommand*{\mailto}[1]{\href{mailto:#1}{\nolinkurl{#1}}}

\makeatletter
\def\theequation{\@arabic\c@equation}

\newcommand{\bb}[1]{{\mathbb{#1}}}

\newcommand{\bbN}{{\mathbb{N}}}
\newcommand{\bbR}{{\mathbb{R}}}

\newcommand{\bbC}{{\mathbb{C}}}

\newcommand{\cB}{{\mathcal B}}

\newcommand{\cH}{{\mathcal H}}

\newcommand{\cS}{{\mathcal S}}

\newcommand{\cW}{{\mathcal W}}

\newcommand{\no}{\nonumber}
\newcommand{\lb}{\label}
\newcommand{\bi}{\bibitem}
\newcommand{\f}{\frac}

\newcommand{\ol}{\overline}
\newcommand{\bs}{\backslash}

\newcommand{\wti}{\widetilde}
\newcommand{\wha}{\widehat}
\newcommand{\dott}{\,\cdot\,}
\newcommand{\la}{\lambda}
\newcommand{\al}{\alpha}
\newcommand{\be}{\beta}
\newcommand{\ga}{\gamma}

\newcommand{\si}{\sigma}
\newcommand{\te}{\theta}

\newcommand{\ran}{\operatorname{ran}}
\newcommand{\dom}{\operatorname{dom}}

\renewcommand{\Im}{\operatorname{Im}}

\newcommand{\loc}{\text{\rm{loc}}}

\newcommand{\ess}{\text{\rm{ess}}}

\numberwithin{equation}{section}

\newtheorem{theorem}{Theorem}[section]
\newtheorem{lemma}[theorem]{Lemma}
\newtheorem{corollary}[theorem]{Corollary}
\newtheorem{hypothesis}[theorem]{Hypothesis}

\theoremstyle{definition}

\newtheorem{remark}[theorem]{Remark}

\begin{document}

\title[Renormalized Oscillation Theory]{Renormalized Oscillation Theory for \\
Hamiltonian Systems}

\author[F.\ Gesztesy]{Fritz Gesztesy}
\address{Department of Mathematics,
Baylor University, One Bear Place \#97328,
Waco, TX 76798-7328, USA}
\email{\mailto{Fritz\_Gesztesy@baylor.edu}}
\urladdr{\url{http://www.baylor.edu/math/index.php?id=935340}}

\author[M.\ Zinchenko]{Maxim Zinchenko}
\address{Department of Mathematics and Statistics,
University of New Mexico, Albuquerque, NM 87131, USA}
\email{\mailto{maxim@math.unm.edu}}
\urladdr{\url{http://www.math.unm.edu/~maxim/}}

\dedicatory{Dedicated with admiration to Barry Simon, mentor and friend, on the occasion of his 70th birthday.}
\date{\today}
\thanks{M.Z. is supported in part by a Simons Foundation grant CGM--281971.}
\thanks{Appeared in {\it Advances Math.} {\bf 311}, 569--597 (2017).}
\subjclass[2010]{Primary: 34B24, 34C10; Secondary: 34L15, 34L05.}
\keywords{Matrix oscillation theory, Hamiltonian systems, Sturm--Liouville and Dirac-type operators, eigenvalue counting in essential spectral gaps.}

\begin{abstract}
We extend a result on renormalized oscillation theory, originally derived for Sturm--Liouville and Dirac-type operators on arbitrary intervals in the context of scalar coefficients, to the case of general Hamiltonian systems with block matrix coefficients. In particular, this contains the cases of general Sturm--Liouville and Dirac-type operators with block matrix-valued coefficients as special cases.

The principal feature of these renormalized oscillation theory results consists in the fact that by replacing solutions by appropriate Wronskians of solutions, oscillation theory now applies to intervals in essential spectral gaps where traditional oscillation theory typically fails.
\end{abstract}

\maketitle


\section{Introduction} \lb{s1}

To set the stage for this paper we briefly recall the essentials of traditional Sturm oscillation theory in the simple, special (yet, representative) case of Dirichlet Schr\"odinger operators on a bounded interval $(a,b)$ and a half-line $(a,\infty)$ in terms of zeros of appropriate solutions, and then turn to renormalized oscillation theory in terms of Wronskians of certain solutions due to \cite{GST96a} before describing the principal new results of this paper obtained for general Hamiltonian systems with block matrix coefficients.

Assuming $a \in \bbR$, suppose that
\begin{equation}
V \in L^1_{\loc}((a,\infty)) \, \text{ is real-valued,}    \lb{1.1}
\end{equation}
and (to avoid having to deal with boundary conditions at infinity in the half-line case) that
\begin{equation}
\text{the differential expression $\tau = -\f{d^2}{dx^2} + V(x)$ is in the limit point case at $\infty$.}  \lb{1.2}
\end{equation}
We introduce the Dirichlet operators $H_{a,b}^D$ in $L^2((a,b))$, $a, b \in \bbR$, $a < b$, and
$H_a^D$ in $L^2((a,\infty))$ via
\begin{align}
& \big(H_{a,b}^D f\big)(x) = - f''(x) + V(x) f(x),    \no \\
& f \in \dom\big(H_{a,b}^D\big) = \big\{g \in L^2((a,b)) \, \big| \, g \in AC([a,b]); \, g(a)=0=g(b); \lb{1.3} \\
& \hspace*{6.4cm} (- g'' + V g) \in L^2((a,b))\big\},    \no
\end{align}
and
\begin{align}
& \big(H_a^D f\big)(x) = - f''(x) + V(x) f(x),    \no \\
& f \in \dom\big(H_a^D\big) = \big\{g \in L^2((a,\infty)) \, \big| \, g \in AC_{\loc}([a,\infty)); \, g(a)=0;
\lb{1.4} \\
& \hspace*{5.85cm} (- g'' + V g) \in L^2((a,\infty))\big\}.    \no
\end{align}
In addition, denote by $P\big((\la_0,\la_1);H_{a,b}^D\big)$ the strongly right-continuous spectral projection of $H_{a,b}^D$ corresponding to the open interval $(\la_0,\la_1) \subset \bbR$, and analogously for $H_a^D$.

Next, let $\lambda \in \bbR$ and $\psi_- (\lambda,\dott)$ be a nontrivial solution of $\tau \psi(\lambda,\dott) = \lambda \psi(\lambda,\dott)$ satisfying the Dirichlet boundary condition at the left endpoint $a$, that is,
\begin{equation}
\psi_-(\lambda, a) = 0.    \lb{1.5}
\end{equation}
(Without loss of generality one can assume that $\psi_- (\lambda,\dott)$ is real-valued.)
We denote by $N_{(c,d)}(\psi_-(\lambda,\dott))$ the number of zeros (necessarily simple) of
$\psi_-(\lambda,\dott)$ in the interval $(c,d) \subseteq (a,b)$.

Then the classical Sturm oscillation theorem associated with $H_{a,b}^D$, $H_a^D$ (cf.\ the discussions in \cite{GST96a}, \cite{Si05}) can be stated as follows:

\begin{theorem} \lb{t1.1}
Assume \eqref{1.1} and \eqref{1.2}, and let $\lambda_0 \in \bbR$. Then,
\begin{equation}
\dim\bigl(\ran\bigl(P\big((- \infty,\lambda_0);H_{a,b}^D\big)\big)\big)
= N_{(a,b)}(\psi_-(\lambda_0,\dott)),       \lb{1.6}
\end{equation}
and
\begin{equation}
\dim\bigl(\ran\bigl(P\big((- \infty,\lambda_0);H_a^D\big)\big)\big)
= N_{(a,\infty)}(\psi_-(\lambda_0,\dott)).        \lb{1.7}
\end{equation}
\end{theorem}

Given the incredible amount of literature on aspects of classical oscillation theory for Sturm--Liouville operators, it is impossible to attempt a fair account of the corresponding literature, so we just refer to
a few of the standard books on the subject such as, \cite[Ch.~8]{CL85}, \cite[Sect.~XIII.7]{DS88},
\cite[Ch.~XI]{Ha82}, \cite[Ch.~8]{Hi69}, \cite[Ch.~X]{In56}, \cite[Ch.~1]{Kr73}, \cite[Sect.~1.3]{LS75},
\cite[Ch.~II--IV]{Re80}, \cite[Ch.~2]{Sw68}, \cite[Sects.~13, 14]{We87}.

In the half-line case \eqref{1.7}, if $\lambda_0 > \inf \sigma_{\ess}\big(H_a^D\big)$, then $\tau$ is
oscillatory at $\lambda_0$ near $\infty$ (i.e., every real-valued solution $u$ of $\tau u = \lambda_0 u$
has infinitely many zeros in $(a,\infty)$ accumulating at $\infty$) and either side in \eqref{1.7}
equals $\infty$. For $\lambda_j \in \bbR$, $j=0,1$, $\lambda_0 < \lambda_1$, with $\tau$ being nonoscillatory at $\lambda_1$ near $a$ (i.e., every real-valued solution $u$ of $\tau u = \lambda_1 u$ has finitely many zeros in $(a,c)$ for every $c \in (a,\infty)$), and nonoscillatory near $\infty$ (i.e., every real-valued solution $u$ of $\tau u = \lambda_1 u$ has finitely many zeros in $(c,\infty)$ for every $c \in (a,\infty)$), then,
\begin{equation}
\dim\bigl(\ran\bigl(P\big([\lambda_0,\lambda_1);H_a^D\big)\big)\big)
= \lim_{c \uparrow \infty} [N_{(a,c)}(\psi_-(\lambda_1,\dott)) -
N_{(a,c)}(\psi_-(\lambda_0,\dott))].    \lb{1.8}
\end{equation}
Similarly, if $\tau$ is nonoscillatory at $\lambda_1$ near $a$ and oscillatory at $\lambda_1$ near
$\infty$, then
\begin{equation}
\dim\bigl(\ran\bigl(P\big((\lambda_0,\lambda_1);H_a^D\big)\big)\big)
= \liminf_{c \uparrow \infty} [N_{(a,c)}(\psi_-(\lambda_1,\dott)) -
N_{(a,c)}(\psi_-(\lambda_0,\dott))].   \lb{1.9}
\end{equation}
These facts are proved in \cite{GST96a}, they represent slight extensions of results of Hartman \cite{Ha49} and motivate the notion of renormalized oscillation theory in the context where
$\lambda_0 > \inf \sigma_{\ess}\big(H_a^D\big)$.

A novel approach to oscillation theory, especially efficient if
$\lambda_0 > \inf \sigma_{\ess}\big(H_a^D\big)$, replacing solutions $\psi_-(\lambda,\dott)$ by appropriate Wronskians of solutions, was introduced in 1996 in \cite{GST96a} (motivated by results in
\cite{FIT87}, \cite{GST96}, and \cite{Le52}). To describe this
result we suppose that $\psi_+(\lambda,\dott)$, $\lambda \in \bbR$, is either a nontrivial
real-valued solution of $\tau \psi(\lambda,\dott) = \lambda \psi(\lambda,\dott)$ satisfying the Dirichlet boundary condition at the right endpoint $b$, that is,
\begin{equation}
\psi_+(\lambda, b) = 0,    \lb{1.10}
\end{equation}
or else, in the half-line case $(a,\infty)$, we consider the Weyl--Titchmarsh solution $\psi_+(z,\dott)$ of $\tau \psi(z,\dott) = z \psi(z,\dott)$, $z \in \bbR\backslash \sigma_{\ess}\big(H_a^D\big)$ uniquely defined up to constant multiples (generally depending on $z$) in such a manner that we assume without loss of generality that $\psi_+(\dott, x)$ is analytic on $\bbC \backslash \sigma\big(H_a^D\big)$, and, upon removing poles, also analytic in a neighborhood of the discrete spectrum of $H_a^D$. In addition, we suppose
that $\psi_+(\lambda,\dott)$ is real-valued for $\lambda \in \bbR \backslash \sigma_{\ess}\big(H_a^D\big)$.

Given $\psi_-(\lambda,\dott)$ and $\psi_+(\mu,\dott)$,
$\lambda, \mu \in \bbR \backslash \sigma_{\ess}\big(H_a^D\big)$, we introduce their Wronskian by
\begin{equation}
W(\psi_-(\lambda,\dott), \psi_+(\mu,\dott))(x) =
\psi_-(\lambda, x) \psi_+'(\mu, x) -
\psi_-'(\lambda, x) \psi_+(\mu, x), \quad x \in [a, \infty),     \lb{1.11}
\end{equation}
and denote by $N_{(c,d)}(W(\psi_-(\lambda,\dott), \psi_+(\mu,\dott)))$ the number of zeros
({\bf not} counting multiplicity) of
$W(\psi_-(\lambda,\dott), \psi_+(\mu,\dott))(\cdot)$ either in the interval $(c,d) \subseteq (a,b)$ if $b \in \bbR$, or in the interval $(c,d) \subseteq (a,\infty)$.

One of the principal results obtained in \cite{GST96a} then can be stated as follows:

\begin{theorem} \lb{t1.2}
Assume \eqref{1.1} and \eqref{1.2}, and let $\lambda_0, \lambda_1 \in \bbR$,
$\lambda_0 < \lambda_1$. Then,
\begin{equation}
\dim\bigl(\ran\bigl(P\big((\lambda_0,\lambda_1);H_{a,b}^D\big)\big)\big)
= N_{(a,b)}(W(\psi_-(\lambda_0,\dott), \psi_+(\lambda_1,\dott))),       \lb{1.12}
\end{equation}
and
\begin{equation}
\dim\bigl(\ran\bigl(P\big((\lambda_0,\lambda_1);H_a^D\big)\big)\big)
= N_{(a,\infty)}(W(\psi_-(\lambda_0,\dott), \psi_+(\lambda_1,\dott))).        \lb{1.13}
\end{equation}
\end{theorem}

We emphasize that Theorem \ref{t1.2} applies, especially to situations where
$(\lambda_0, \lambda_1)$ lies in an essential spectral gap of $H_a^D$,
$(\lambda_0, \lambda_1) \subset \bbR \backslash \sigma_{\ess}\big(H_a^D\big)$,
$\lambda_0 > \inf \sigma_{\ess}\big(H_a^D\big)$, a case in which both, $\psi_-(\lambda_0,\dott)$ and $\psi_+(\lambda_1,\dott)$ have infinitely many zeros on $[0,\infty)$.

Reference \cite{GST96a} also contains results with $\psi_+(\lambda_1,\dott)$ replaced by
$\psi_-(\lambda_1,\dott)$, and other extensions, particularly, to self-adjoint, separated
boundary conditions, but we omit further details here. In addition, extensions of Theorem \ref{t1.2},
as well as the treatment of Dirac-type operators and that of the finite difference case of Jacobi operators
appeared in \cite{AT09}, \cite{KT08}, \cite{KT08a}, \cite{KT09}, \cite{ST10}, \cite{Te95}--\cite{Te98},
\cite[Ch.~4]{Te00}.

Although only indirectly related to \eqref{1.12}, we here mention the results obtained in \cite{GG91} connecting the sign changes of the modified Fredholm determinant of a certain Hilbert--Schmidt operator with a semi-separable integral kernel depending on an energy parameter $\lambda_0 \in \bbR$ and the number of eigenvalues of a Sturm--Liouville operator less than $\lambda_0$ on a compact interval with separated boundary conditions. This can be viewed as a continuous analog of the Jacobi--Sturm rule counting the negative eigenvalues of a self-adjoint matrix.

Next, we turn to the principal topic of this paper, extensions of these oscillation theory results to
the case of matrix-valued coefficients $V$. Assuming $m \in \bbN$, we replace condition \eqref{1.1}
now by
\begin{equation}
V \in L^1_{\loc}((a,\infty))^{m \times m}, \, \text{ $V(x)$ is self-adjoint for a.e.~$x \in (a,\infty)$,}
\lb{1.14}
\end{equation}
still supposing that
\begin{equation}
\text{the differential expression $\tau = -\f{d^2}{dx^2}I_m + V(x)$ is in the limit point case at $\infty$.} \lb{1.15}
\end{equation}
Assuming that
$\Psi_-(\lambda,\dott) \in \bbC^{m \times m}$ is a fundamental matrix of solutions of
$\tau \Psi(\lambda,\dott) = \lambda \Psi(\lambda,\dott)$, $\lambda \in \bbR$, satisfying the Dirichlet boundary condition at the left endpoint $a$,
\begin{equation}
\Psi_-(\lambda, a) = 0,   \lb{1.15a}
\end{equation}
and defining $H_{a,b}^D$ and $H_a^D$ in analogy to \eqref{1.3} and \eqref{1.4}, we now denote,
\begin{equation}
N_{(c,d)}(\Psi_-(\lambda,\dott)) := \sum_{x \in (c,d)} \dim(\ker(\Psi_-(\lambda,x))),   \lb{1.16}
\end{equation}
for $(c,d) \subseteq (a,\infty)$. The analog of Theorem \ref{t1.1} in the present matrix context, as derived
in \cite{RH77}, \cite[Ch.~1]{RK05}, then reads as follows:

\begin{theorem} \lb{t1.3}
Assume \eqref{1.14} and \eqref{1.15}, and let $\lambda_0 \in \bbR$. Then,
\begin{equation}
\dim\bigl(\ran\bigl(P\big((- \infty,\lambda_0);H_{a,b}^D\big)\big)\big)
= N_{(a,b)}(\Psi_-(\lambda_0,\dott)),       \lb{1.17}
\end{equation}
and if $\lambda_0 \leq \inf\big(\si_{\rm ess}\big(H_a^D\big)\big)$,
\begin{align}
\begin{split}
\dim\bigl(\ran\bigl(P\big((- \infty,\lambda_0);H_a^D\big)\big)\big) &=
\lim_{b \uparrow \infty}
\dim\bigl(\ran\bigl(P\big((-\infty,\lambda_0);H_{a,b}^D\big)\big)\big)   \\
&= N_{(a, \infty)}(\Psi_-(\lambda_0,\dott)).        \lb{1.18}
\end{split}
\end{align}
\end{theorem}

Also the amount of available literature on oscillation theory, disconjugacy theory, rotation numbers, etc.,  in the context of matrix-valued Sturm--Liouville operators and more generally, Hamiltonian systems with block matrix coefficients,
is far too numerous to be accounted for at this point. We thus just confine ourselves to a few pertinent references in this context such as, \cite[Ch.~10]{At64}, \cite{CGN16}, \cite[Ch.~2]{Co71}, \cite{Et70},
\cite{FJN03}--\cite{FJNN11}, \cite{GJ89}, \cite{Ha57}, \cite[Sects.~XI.10, XI.11]{Ha82}, \cite[Sect.~9.6]{Hi69}, \cite{Ja64}, \cite[Ch.~2]{JONNF16}, \cite[Chs.~4, 7]{Kr95}, \cite[Ch.~V]{Re80}. In spite of this wealth of results in oscillation theory in the matrix-valued context, it appears that the precise connection between oscillation and spectral properties contained in Theorem \ref{t1.3} is not covered by these sources, but goes back to \cite{RH77} (see also \cite[Ch.~1]{RK05}). In addition, we note that \cite[pp.~367--368]{RH77} briefly discusses the fact that results of the type Theorem \ref{t1.3} include the Morse index theorem (in this context see also \cite{GJ89}).

As in the context of Theorem \ref{t1.1}, Theorem \ref{t1.3} permits various extensions, particularly to
other self-adjoint, separated boundary conditions, etc. Therefore,  we omit further details at this point as we will treat a very general case in the main body of this paper.

While Theorem \ref{t1.3} is as close as possible to a matrix-valued analog of the celebrated classical scalar oscillation result, Theorem \ref{t1.1}, the analog of Theorem \ref{t1.2} in the matrix context remained an open problem since 1996. It is precisely this problem that will be settled in this paper. In fact, we will not only treat the case of Schr\"odinger (actually, general, three-coefficient Sturm--Liouville) operators and Dirac-type operators with matrix-valued coefficients (cf.\ \eqref{1.29}--\eqref{1.32} below), but the more general case of finite interval, half-line, and full-line Hamiltonian (also called, canonical) systems of the form,
\begin{equation}
J \Psi'(z,x) = [z A(x) + B(x)] \Psi(z,x), \quad x \in
\begin{cases}[a,b], & \text{if $- \infty < a < b < \infty$}, \\
[a,b), & \text{if $- \infty < a < b = \infty$}, \\
\bbR, & \text{if $(a,b) = \bbR$}, \end{cases}  \quad z \in \bbC,          \lb{1.19}
\end{equation}
for solutions $\Psi(z,\dott) \in \bbC^{2m \times \ell}$, $\ell \in \bbN$, $1 \leq \ell \leq 2m$,  satisfying
\begin{equation}
\Psi \in \begin{cases} AC([a,b])^{2m \times \ell}, & \text{if $- \infty < a < b < \infty$,} \\
AC_\loc([a,b))^{2m \times \ell}, & \text{if $-\infty < a < b = \infty$,} \\
AC_\loc(\bbR)^{2m \times \ell}, & \text{if $(a,b) = \bbR$}. \end{cases}    \lb{1.20}
\end{equation}
Here, $J = \left(\begin{smallmatrix} 0_m & - I_m \\ I_m & 0_m\end{smallmatrix}\right)$, $m \in \bbN$,
where $I_m$ is the identity matrix and $0_m$ is the zero matrix in
$\bbC^{m \times m}$, and given $r \in \bbN$, $1 \leq r \leq 2m$,
\begin{align}
\begin{split}
& 0 \leq A(x) \in \bbC^{2m \times 2m}, \quad
A(x) = \begin{pmatrix} W(x) & 0 \\ 0 & 0 \end{pmatrix},
\quad 0 < W(x) \in \bbC^{r \times r}, \\
& B(x) = B^*(x) \in \bbC^{2m \times 2m}     \lb{1.21}
\end{split}
\end{align}
for a.e.\ $x \in (a,b)$, with locally integrable entries as described in \eqref{2.2}--\eqref{2.4} and we assume again the
limit point case at $\pm \infty$.

Given the Hamiltonian system \eqref{1.19}, introducing
$E_r = \left(\begin{smallmatrix} I_r & 0 \\ 0 & 0 \end{smallmatrix}\right) \in \bbC^{2m \times 2m}$, one can introduce associated operators $T_{a,b}$, $T_a$, and $T$ in the finite interval, half-line, and
full-line case, mapping a subset of $L^2_A((a,b))^{2m}$ into $E_r L^2_A((a,b))^{2m}$, respectively, according to \eqref{2.21}--\eqref{2.23}. Here the space $L^2_A((c,d))^{2m}$ is introduced in \eqref{2.12}--\eqref{2.14}. For matters of brevity and simplicity, we confine ourselves for the
remainder of this introduction to the half-line case $- \infty < a < b=\infty$.

In addition, for $z \in \bbC \bs \bbR$ and a fixed reference point $x_0 \in (a,\infty)$, one can introduce appropriate Weyl--Titchmarsh solutions $\Psi_{-, \al}(z,\dott,x_0) \in \bbC^{2m \times m}$ and
$\Psi_{+}(z,\dott,x_0) \in \bbC^{2m \times m}$ of \eqref{1.19}, where $\Psi_{-, \al}(z,\dott,x_0)$ satisfies the self-adjoint $\al$-boundary condition at $x=a$,
\begin{equation}
\al^* J \Psi_{-,\al}(z, a,x_0) = 0,
\end{equation}
and $\Psi_{+}(z,\dott,x_0)$ satisfies for all $c \in (a, \infty)$,
\begin{equation}
\Psi_{+}(z,\dott,x_0) \in L^2_A((c,\infty))^{2m}.
\end{equation}
Here the boundary condition matrix $\alpha \in \bbC^{2m \times m}$ satisfies \eqref{2.17}, and the reference point $x_0$ is used
to introduce a convenient normalization of $\Psi_{-, \al}(z,\dott,x_0)$ and
$\Psi_{+}(z,\dott,x_0)$ as discussed in \eqref{2.47}--\eqref{2.50}.

Recalling a special case of the Wronskian-type identity for solutions
$\Psi(\lambda_j,\dott) \in \bbC^{2m \times \ell}$, $1 \leq \ell \leq 2m$, $\lambda_j \in \bbR$, $j=0,1$,
of \eqref{1.19},
\begin{align}
\f{d}{dx} [\Psi(\lambda_0,x)^* J \Psi(\lambda_1,x)] &= (\lambda_1 - \lambda_0)
\Psi(\lambda_0,x)^* A(x) \Psi(\lambda_1,x),  \quad x \in (a,b)  \lb{1.25}
\end{align}
(cf.\ \eqref{2.20}), we now denote
\begin{align}
\begin{split}
& N_{(c,d)}(\Psi_+(\la_0,\dott,x_0)^* J \Psi_{-,\al}(\la_1,\dott,x_0))   \\
& \quad := \sum_{x \in (c,d)} \dim(\ker(\Psi_+(\la_0,x,x_0)^*
J \Psi_{-,\al}(\la_1, x,x_0))),
\end{split}
\end{align}
for $(c,d) \subseteq (a,\infty)$. (In the special case of matrix-valued Schr\"odinger operators, see, \eqref{1.29}--\eqref{1.30}, by appropriately partitioning $\Psi_{-,\al}, \Psi_+$ into $m \times m$ blocks, one readily verifies that $\Psi(\lambda_0,x)^* J \Psi(\lambda_1,x)$ corresponds precisely to the Wronskian of $m \times m$ matrix-valued solutions of Schr\"odinger's equation in analogy to \eqref{1.11}.)

Finally, we also introduce the symbol $N((\la_0, \la_1); T_a)$ to denote the sum of geometric multiplicities of all eigenvalues of $T_a$ in the interval $(\la_0, \la_1)$.

Then our principal new result in the matrix-valued context, formulated in the special half-line
case (cf.\ Theorem \ref{t3.10}), and a direct analog of the scalar half-line case, \eqref{1.13}
in Theorem \ref{t1.2}, reads as follows:

\begin{theorem} \lb{t1.4}
Assume Hypotheses \ref{h2.2}, \ref{h3.1}, $\la_0, \la_1 \in \bbR \bs \si(T_a)$, $\la_0 < \la_1$, and
$(\la_0,\la_1)\cap\si_\ess(T_a)=\emptyset$. Then,
\begin{equation}
N((\la_0,\la_1); T_a) = N_{(a,\infty)}(\Psi_+(\la_0,\dott,x_0)^* J \Psi_{-,\al}(\la_1,\dott,x_0)).
\lb{1.28}
\end{equation}
\end{theorem}

We emphasize that the interval $(\la_0,\la_1)$ can lie in any essential spectral gap of $T_a$, not just below its essential spectrum as in standard approaches to oscillation theory in the matrix-valued context.
Extensions to the finite interval as well as full-line cases will be discussed in the main body of this paper.
Moreover, these types of oscillation results for general Hamiltonian systems, to the best of our knowledge, appear to be new even in the special scalar case $m=1$.

Without entering details, we note that the new strategy of proof in this matrix-valued extension of the 1996 scalar oscillation theory result in \cite{GST96a} differs from the one originally employed in \cite{GST96a} and now rests to a large extent on approximations of a given operator by appropriate restrictions.

Finally, to demonstrate the well-known fact that three-coefficient Sturm--Liouville as well as Dirac-type operators are included in Hamiltonian systems of the form \eqref{1.19} as special cases, it suffices to recall the following observations: The $m \times m$ matrix-valued Sturm--Liouville differential expression
\begin{equation}
R(x)^{-1} [- (d/dx) P(x) (d/dx) + Q(x)],    \lb{1.29}
\end{equation}
with $P(x), Q(x), R(x) \in \bbC^{m \times m}$, $m \in \bbN$, appropriate positivity hypotheses
on $P, R$, and local integrability of $P^{-1}, Q, R$, subordinates to the Hamiltonian system
\eqref{1.19} with the choice
\begin{equation}
A(x) =\begin{pmatrix} R(x) & 0_m \\[1mm] 0_m & 0_m \end{pmatrix}, \quad
B(x) = \begin{pmatrix} - Q(x) & 0_m \\[1mm] 0_m & P(x)^{-1} \end{pmatrix}.    \lb{1.30}
\end{equation}
Similarly, the Dirac-type differential expression
\begin{equation}
J (d/dx) - B(x),    \lb{1.31}
\end{equation}
with $B(x) \in \bbC^{2m \times 2m}$ and locally integrable entries, simply corresponds to \eqref{1.19} with the choice
\begin{equation}
A(x) =I_{2m}.     \lb{1.32}
\end{equation}

At this point we briefly turn to the content of each section: Section \ref{s2} recalls the basics of Hamiltonian systems as needed in this paper and proves a few additional facts in this context that appear to be new. Renormalized oscillation theory on a half-line is discussed in detail in
Section \ref{s3}. (The treatment of a finite interval is a simple special case of the half-line case.) The principal result, Theorem \ref{t3.10}, coincides with Theorem \ref{t1.4} above. The extension to the full line case is developed in our final Section \ref{s4}.

Finally, we briefly comment on the notation used in this paper: Throughout, $\cH$
denotes a separable, complex Hilbert space with inner product and norm denoted by $(\dott,\dott)_{\cH}$ (linear in the second argument) and $\|\cdot \|_{\cH}$, respectively. The identity operator in $\cH$ is written as $I_{\cH}$. We denote by $\cB(\cH)$ (resp., $\cB_{\infty}(\cH)$) the Banach space of linear bounded (resp., compact) operators in $\cH$. The domain, range, kernel (null space), and spectrum of a linear operator will be denoted by $\dom(\cdot)$, $\ran(\cdot)$, $\ker(\cdot)$, and $\si(\cdot)$, respectively. For a self-adjoint operator $A$ in $\cH$, $P((\la_0,\la_1); A)$ denotes the strongly
right-continuous spectral projection of $A$ associated to the open interval $(\la_0,\la_1) \subset \bbR$.

The space of $k \times \ell$ matrices with complex-valued entries is denoted by
$\bbC^{k \times \ell}$, or simply by $\bbC^k$ if $\ell = 1$.
The symbol $I_k$ represents the identity matrix in $\bbC^{k \times k}$.
The shorthand notation $L^p((a,b))^{k \times \ell} := L^p ((a,b), dx; \bbC^{k \times \ell})$, $p \geq 1$, $k, \ell \in \bbN$, and for its variants with $(a,b)$ replaced by $[a,b)$ and/or $\bbR$ as well as in the case of local integrability, will be used. The superscript $\ell$ is again dropped if $\ell = 1$.
We employ the same conventions to (locally) absolutely continuous functions replacing $L^p$ by $AC$. In particular, we use the convention,
$AC_\loc([a,\infty)) = \big\{\phi \in AC([a,c]) \, \text{for all $c > a$}\big\}$.

\section{Basic Facts on Hamiltonian Systems} \label{s2}

In this section we recall the basic results on a class of Hamiltonian systems on arbitrary
intervals. For basic results on Hamiltonian systems we will employ in this paper we refer, for
instance, to \cite[Chs.~9--10]{At64}, \cite[Ch.~2]{Co71}, \cite[Sects.~XI.10, XI.11]{Ha82},
\cite{HS81}--\cite{HS86},
\cite{Kr89a}, \cite{Kr89b}, \cite[Chs.~4, 7]{Kr95}, \cite{LM03}, see also \cite{JONNF16} for a
most recent treatment of oscillation, spectral, and control theory for Hamiltonian systems.

\begin{hypothesis} \lb{h2.1}
Fix $m \in \bbN$ and introduce the $2m \times 2m$ matrix
\begin{equation}
J = \begin{pmatrix} 0_m & - I_m \\ I_m & 0_m\end{pmatrix}    \lb{2.1}
\end{equation}
where $I_m$ is the identity matrix and $0_m$ is the zero matrix in
$\bbC^{m \times m}$. Let $-\infty\leq a < b\leq\infty$ and fix $r \in \bbN$
such that $1 \leq r \leq 2m$. Assume $($for a.e.~$x \in (a,b)$$)$
\begin{equation}
0 < W(x) \in \bbC^{r \times r}, \quad W \in
\begin{cases}
L^1([a,b])^{r \times r}, & \text{if $-\infty < a < b < \infty$,} \\
L^1_\loc([a,b))^{r \times r}, & \text{if $-\infty < a < b = \infty$,}\\
L^1_\loc(\bbR)^{r \times r}, & \text{if $(a,b)=\bbR$,}
\end{cases}     \lb{2.2}
\end{equation}
and introduce $($again for a.e.~$x \in (a,b)$$)$
\begin{align}
\begin{split}
& 0 \leq A(x), C(x), E_r \in \bbC^{2m \times 2m}, \\
& A(x) = \begin{pmatrix} W(x) & 0 \\ 0 & 0 \end{pmatrix}, \quad
C(x) = \begin{pmatrix} W(x)^{-1} & 0 \\ 0 & I_{2m-r} \end{pmatrix}, \quad
E_r = \begin{pmatrix} I_r & 0 \\ 0 & 0 \end{pmatrix},    \lb{2.3}
\end{split}
\end{align}
so that $C A = A C = E_r$. In addition, assume $($once more
for a.e.~$x \in (a,b)$$)$
\begin{equation}
B(x) = B(x)^* \in \bbC^{2m \times 2m}, \quad B \in
\begin{cases}
L^1([a,b])^{2m \times 2m}, & \text{if $-\infty < a < b < \infty$,} \\
L^1_\loc([a,b))^{2m \times 2m}, & \text{if $-\infty < a < b = \infty$,} \\
L^1_\loc(\bbR)^{2m \times 2m}, & \text{if $(a,b)=\bbR$.}
\end{cases}      \lb{2.4}
\end{equation}
\end{hypothesis}

Granted the matrices $A, B$, and depending on whether
$a$ and/or $b$ are finite, we consider  Hamiltonian systems of the form,
\begin{equation}
J \Psi'(z,x) = [z A(x) + B(x)] \Psi(z,x), \quad x \in
\begin{cases}[a,b], & \text{if $- \infty < a < b < \infty$}, \\
[a,b), & \text{if $- \infty < a < b = \infty$}, \\
\bbR, & \text{if $(a,b) = \bbR$}, \end{cases}  \quad z \in \bbC,          \lb{2.10}
\end{equation}
for solutions $\Psi(z,\dott) \in \bbC^{2m \times \ell}$, $\ell \in \bbN$, $1 \leq \ell \leq 2m$,  satisfying
\begin{equation}
\Psi \in \begin{cases} AC([a,b])^{2m \times \ell}, & \text{if $- \infty < a < b < \infty$,} \\
AC_\loc([a,b))^{2m \times \ell}, & \text{if $-\infty < a < b = \infty$,} \\
AC_\loc(\bbR)^{2m \times \ell}, & \text{if $(a,b) = \bbR$}. \end{cases}
\end{equation}

Let $c<d$. It is convenient to introduce the Hilbert space
\begin{align}
L^2_W((c,d))^{r} :=\big\{ F: (c,d) \to \bbC^{r} \, \text{measurable} \,:\, \|F\|_{L^2_W((c,d))^{r}}<\infty \big\},     \lb{2.12}
\end{align}
with the norm
\begin{equation}
\|F\|_{L^2_W((c,d))^{r}}^2 :=\int_c^d dx \, F(x)^* W(x) F(x).   \lb{2.13}
\end{equation}
In addition, we introduce the natural restriction operator $\wha E_r: \bbC^{2m} \to \bbC^r$ and the space
\begin{align}
L^2_A((c,d))^{2m}= \big\{F: (c,d) \to \bbC^{2m} \, \text{measurable} \, \big|\, \wha E_r F \in L^2_W((c,d))^{2m}\big\},    \lb{2.14}
\end{align}
with the seminorm $\|F\|_{L^2_A((c,d))^{2m}} = \|\wha E_r F\|_{L^2_W((c,d))^{r}}$.

In order to be able to discuss boundary conditions at $a,b$ if the latter are finite we now introduce a class of matrices $\al = (\al_1\;\; \al_2)^\top := \binom{\al_1}{\al_2} \in \bbC^{2m \times m}$ satisfying
that
\begin{equation}
(\al \;\; J\al) = \begin{pmatrix} \al_1 & - \al_2 \\ \al_2 & \al_1  \end{pmatrix} \,
\text{ is a unitary $\bbC^{2m \times 2m}$ matrix.}   \lb{2.17}
\end{equation}
Explicitly, \eqref{2.17} reads
\begin{align}
& \al_1 \al_1^* + \al_2 \al_2^* = I_m = \al_1^* \al_1 + \al_2^* \al_2,
\quad
\al_1 \al_2^* = \al_2 \al_1^*,
\quad
\al_1^* \al_2 = \al_2^* \al_1.    \lb{2.18}
\end{align}
We also point out that \eqref{2.17} is equivalent to
\begin{equation}
\al^* \al = I_m, \quad \al^* J \al = 0_m, \quad
\al \al^* - J \al \al^* J = I_{2m}.     \lb{2.19}
\end{equation}

From this point on, if $b = \infty$ (resp., $a = - \infty$), we will always assume the limit point case at $\infty$ (resp., $- \infty$). (We recall that the limit point case at $\infty$ (resp. $- \infty$) is known to be equivalent to the fact that for all $z \in \bbC \bs \bbR$, $c \in (a,b)$, the dimension of all $L^2_A([c, \infty))^{2m}$-solutions (resp., $L^2_A((-\infty,c])^{2m}$-solutions) of \eqref{2.10} equals $m$.) Moreover, we will always assume that all solutions $\Psi$ of \eqref{2.10} with $\ell = 1$ satisfy Atkinson's {\it definiteness condition} in the form below:

\begin{hypothesis} \lb{h2.2}
Assume Hypothesis \ref{h2.1}. \\
$(i)$ Suppose that for all $c, d \in (a,b)$ with $a < c < d < b$, any nonzero solution $\Psi \in AC([a,b])^{2m}$ if $- \infty < a < b < \infty$, $\Psi \in AC_\loc([a,\infty))^{2m}$ if $- \infty < a < b = \infty$, or
$\Psi \in AC_\loc(\bbR)^{2m}$ if $(a,b) = \bbR$, of \eqref{2.10} satisfies
\begin{equation}
\|\chi_{[c,d]} \Psi\|_{L^2_A((a,b))^{2m}} > 0.
\end{equation}
$(ii)$ The boundary condition matrices $\al, \be, \ga \in \bbC^{2m \times m}$ corresponding to $a > - \infty$, $b < \infty$, and $c \in (a,b)$, respectively, are assumed to satisfy \eqref{2.17} $($equivalently, \eqref{2.18}, \eqref{2.19}$)$. \\
$(iii)$ If $b = \infty$ $($resp., $a = - \infty$$)$, we assume the limit point case at $\infty$ $($resp., $- \infty$$)$.
\end{hypothesis}

One recalls the Wronskian-type identity for solutions
$\Psi(z_j,\dott) \in \bbC^{2m \times \ell}$, $1 \leq \ell \leq 2m$, $z_j \in \bbC$, $j=1,2$, of the inhomogeneous equation,
\begin{equation}
J\Psi'(z_j,x) = [z_j A(x) + B(x)] \Psi(z,x) + D_j(x), \quad x \in (a,b), \; j=1,2,
\end{equation}
with $D_j \in L^1_\loc((a,b))^{2m \times \ell}$,
\begin{align}
\f{d}{dx} [\Psi(z_1,x)^* J \Psi(z_2,x)] &= (z_2 - \ol{z_1}) \Psi(z_1,x)^* A(x) \Psi(z_2,x) \lb{2.20} \\
& \quad + \Psi(z_1,x)^* D_2 (x) - D_1 (x)^* \Psi(z_2,x),  \quad x \in (a,b). \no
\end{align}

Given boundary matrices $\al$ and $\be$ satisfying \eqref{2.17} (equivalently,
\eqref{2.18}, \eqref{2.19}), we define the matrix-valued differential expression $\tau$ by
\begin{equation}
\tau = C(x) \bigg[J \f{d}{dx} - B(x)\bigg], \quad x \in (a,b).
\end{equation}
An operator $T_{a,b} : \dom(T_{a,b}) \to E_r L^2_A((a,b))^{2m}$ if
$- \infty < a < b < \infty$, $T_a: \dom(T_a) \to E_r L^2_A((a,\infty))^{2m}$ if
$b = \infty$, and $T: \dom(T) \to E_r L^2_A(\bbR)^{2m}$ if $(a,b) = \bbR$, associated to the Hamiltonian system \eqref{2.10} is then introduced as follows:
\begin{align}
& T_{a,b} F = \tau F,    \no \\
& F \in  \dom(T_{a,b}) = \big\{G \in L^2_A((a,b))^{2m} \, \big| \,
G \in AC([a,b])^{2m}; \, \al^* J G(a) = 0 = \be^* J G(b);   \no \\
& \hspace*{5.7cm} \, \tau G \in E_r L^2_A((a,b))^{2m} \big\},   \lb{2.21} \\
& T_a F = \tau F,    \no \\
& F \in  \dom(T_a) = \big\{G \in L^2_A((a,\infty))^{2m} \, \big| \,
G \in AC_\loc([a,\infty)^{2m}; \, \al^* J G(a) = 0;   \no \\
& \hspace*{5.7cm} \, \tau G \in E_r L^2_A((a,\infty))^{2m} \big\},   \lb{2.22} \\
& T F = \tau F,   \lb{2.23}\\
& F \in  \dom(T) = \big\{G \in L^2_A(\bbR)^{2m} \, \big| \,
G \in AC_\loc(\bbR)^{2m}; \, \tau G \in E_r L^2_A(\bbR)^{2m} \big\}.   \no
\end{align}

The boundary condition $\al^* J G(a) = 0$ can be seen to be equivalent to that  discussed, for instance, in \cite[p.\ 319]{HS81} choosing
\begin{equation}
M = \begin{pmatrix} 0_m & \al_1 \\ 0_m & \al_2 \end{pmatrix},
\end{equation}
and similarly, $\be^* J G(b) = 0$ corresponds to the choice
\begin{equation}
N = \begin{pmatrix} \be_1 & 0_m \\ \be_2 & 0_m \end{pmatrix}
\end{equation}
in \cite[p.\ 319]{HS81} so that $M^* J M = 0_{2m} = N^* J N$ and
$\ker (M) \cap \ker (N) = \{0\}$.

As discussed in \cite[Sect.\ 2]{HS82}, for $z \in \bbC \bs \bbR$, one has
bijections
\begin{align}
(T_{a,b} -z E_r) &: \dom(T_{a,b}) \to E_r L^2_A((a,b))^{2m},   \\
(T_a -z E_r) &: \dom(T_a) \to E_r L^2_A((a,\infty))^{2m},   \\
(T -z E_r) &: \dom(T) \to E_r L^2_A(\bbR)^{2m},
\end{align}
and the estimates
\begin{align}
&\big\|(T_{a,b} - z E_r)^{-1} F\big\|_{L^2_A((a,b))^{2m}} \leq
|\Im(z)|^{-1} \|F\|_{L^2_A((a,b))^{2m}}, \quad F \in E_r L^2_A((a,b))^{2m},    \no \\
&\big\|(T_a - z E_r)^{-1} F\big\|_{L^2_A((a,\infty))^{2m}} \leq
|\Im(z)|^{-1} \|F\|_{L^2_A((a,\infty))^{2m}}, \quad F \in E_r L^2_A((a,\infty))^{2m},   \no \\
&\big\|(T - z E_r)^{-1} F\big\|_{L^2_A(\bbR)^{2m}} \leq
|\Im(z)|^{-1} \|F\|_{L^2_A(\bbR)^{2m}}, \quad F \in E_r L^2_A(\bbR)^{2m}.
\lb{2.30}
\end{align}

Following \cite[Sect.\ 2]{HS82}, the spectrum, $\si(T_{a,b})$ of $T_{a,b}$ consists of those $\la \in \bbC$ such that $(T_{a,b} - \la E_r)$ has no bounded inverse, and analogously for $\si(T_a)$ and $\si(T)$.
In particular, $\la \in \si_p(T_{a,b})$ (resp., $\la \in \si_p(T_a)$ or $\la \in \si_p(T$)) if and only if there exists $\Psi \in \dom(T_{a,b})$ (resp., $\Psi \in \dom(T_a)$ or $\Psi \in \dom(T)$) such that
\begin{equation}
(T_{a,b} - \la E_r) \Psi = 0 \quad
(\text{resp., } \, (T_a - \la E_r) \Psi = 0 \, \text{ or } \,
(T - \la E_r) \Psi = 0).
\end{equation}
By the estimates \eqref{2.30},
\begin{equation}
\si(T_{a,b}) \subseteq \bbR, \quad \si(T_a) \subseteq \bbR,
\quad \si(T) \subseteq \bbR.
\end{equation}
In the case $- \infty < a < b < \infty$ the spectrum $\si(T_{a,b})$ is purely discrete,
$\si(T_{a,b}) = \si_d(T_{a,b})$, that is, it consists of isolated eigenvalues only.

Next, employing the adjoint of $\wha E_r$, the extension operator by zero, $\wha E_r^* : \bbC^r \to \bbC^{2m}$ and recalling that $\wha E_r L^2_A((a,b))^{2m} = L^2_W((a,b))^r$, we introduce the restricted resolvents,
\begin{align}
R_{a,b}(z) &:= \wha E_r (T_{a,b} - z E_r)^{-1} \wha E_r^*
\in \cB\big(L^2_W((a,b))^r\big), \quad z\in\bbC\bs\si(T_{a,b}),  \no
\\
R_a(z) &:= \wha E_r (T_a - z E_r)^{-1} \wha E_r^*
\in \cB\big(L^2_W((a,\infty))^r\big), \quad z\in\bbC\bs\si(T_a),  \lb{2.31}
\\
R(z) &:= \wha E_r (T - z E_r)^{-1} \wha E_r^*
\in \cB\big(L^2_W(\bbR)^r\big), \quad z\in\bbC\bs\si(T).    \no
\end{align}
Of importance in the sequel will be a spectral mapping result of the following form.

\begin{lemma} \lb{l2.3}
Suppose $\la_0\in\bbC\bs\si(T)$. Then $\la_1\in\si_p(T)$ if and only if $(\la_1-\la_0)^{-1}\in\si_p(R(\la_0))$ and the geometric multiplicities of $\la_1$ and $(\la_1-\la_0)^{-1}$ are equal. In addition, $\la_1\in\bbC\bs\si(T)$ if and only if $(\la_1-\la_0)^{-1}\in\bbC\bs\si(R(\la_0))$. Analogous results also hold for $T_a$, $R_a$ and $T_{a,b}$, $R_{a,b}$.
\end{lemma}
\begin{proof}
First, suppose $\la_1\in\si_p(T)$ is of geometric multiplicity $n$. In this case there exist $\{\Psi_j\}_{j=1}^n\subset\dom(T)$ such that $\{\Psi_j\}_{j=1}^n$ are linearly independent in $L^2_A(\bbR)^{2m}$ and $T\Psi_j = \la_1 E_r\Psi_j$, $j=1,\dots,n$. Subtracting $\la_0 E_r \Psi_j$ from both sides of the last identity and rearranging yield
\begin{align}
(T - \la_0 E_r)^{-1} E_r \Psi_j = (\la_1 - \la_0)^{-1} \Psi_j, \quad j=1,\dots,n,
\end{align}
and hence, by $E_r = \wha E_r^* \wha E_r$,
\begin{align}
R(\la_0)\wha E_r \Psi_j = (\la_1 - \la_0)^{-1} \wha E_r\Psi_j, \quad j=1,\dots,n.
\end{align}
Since $\|\wha E_r\Psi\|_{L^2_W(\bbR)^{r}}=\|\Psi\|_{L^2_A(\bbR)^{2m}}$ for any $\Psi\in L^2_A(\bbR)^{2m}$, one concludes that $\{\wha E_r \Psi_j\}_{j=1}^n$ are linearly independent in $L^2_W(\bbR)^{r}$, and hence, $(\la_1-\la_0)^{-1}\in\si_p(R(\la_0))$ is of geometric multiplicity $k\geq n$.

Conversely, suppose $(\la_1-\la_0)^{-1}\in\si_p(R(\la_0))$ is of geometric multiplicity $k$. In this case there exist linearly independent $\{\wha\Psi_j\}_{j=1}^k\subset L^2_W(\bbR)^{r}$ such that
$R(\la_0)\wha\Psi_j = (\la_1-\la_0)^{-1}\wha\Psi_j$, $j=1,\dots,k$. Define $\Psi_j = (\la_1-\la_0)(T-\la_0E_r)^{-1}\wha E_r^*\wha\Psi_j$, then one has $\Psi_j\in\dom(T)$ and
\begin{align} \lb{2.31a}
\wha E_r\Psi_j = (\la_1-\la_0)R(\la_1)\wha\Psi_j = \wha\Psi_j, \quad j=1,\dots,k.
\end{align}
Multiplying the last identity by $\wha E_r^*$ and recalling that $E_r=\wha E_r^*\wha E_r$ then yield $E_r\Psi_j=\wha E_r^*\wha\Psi_j$, $j=1,\dots,k$. Thus, one obtains from the definition of $\Psi_j$ that
\begin{align}
(T-\la_0 E_r)\Psi_j = (\la_1-\la_0)\wha E_r^*\wha\Psi_j = (\la_1-\la_0) E_r\Psi_j, \quad j=1,\dots,k,
\end{align}
and hence, $T\Psi_j = \la_1 E_r\Psi_j$, $j=1,\dots,k$. Since $\|\Psi\|_{L^2_A(\bbR)^{2m}} = \|\wha E_r\Psi\|_{L^2_W(\bbR)^{r}}$ for any $\Psi\in L^2_A(\bbR)^{2m}$, one concludes from \eqref{2.31a} that $\{\Psi_j\}_{j=1}^k$ are linearly independent in $L^2_A(\bbR)^{2m}$, and hence, $\la_1\in\si_p(T)$ is of geometric multiplicity $n\geq k$. The two opposite inequalities yield equality of geometric multiplicities $n=k$.

Next, suppose $\la_1\in\bbC\bs\si(T)$. Then one has for $j=0,1$,
\begin{align}
&(T-\la_j E_r)^{-1}(T-\la_j E_r)F=F, \quad F\in\dom(T), \lb{2.32}
\\
&(T-\la_j E_r)(T-\la_j E_r)^{-1}F=F, \quad F\in E_r L^2_A(\bbR)^{2m},   \lb{2.33}
\end{align}
and hence
\begin{align}
(T-\la_0 E_r)(T-\la_1 E_r)^{-1}F = F + (\la_1-\la_0) E_r (T-\la_1 E_r)^{-1}F, \quad F\in E_r L^2_A(\bbR)^{2m}.   \lb{2.34}
\end{align}
Since $E_r L^2_A(\bbR)^{2m}=\wha E_r^* L^2_W(\bbR)^r$, it follows from \eqref{2.34} that the operator
\begin{align}
S=\wha E_r (T-\la_0 E_r)(T-\la_1 E_r)^{-1} \wha E_r^*,
\end{align}
is bounded on $L^2_W(\bbR)^r$. In addition, it follows from \eqref{2.34}, \eqref{2.31}, and $\wha E_r\wha E_r^* = I_{L^2_W(\bbR)^r}$ that
\begin{align}
\wha E_r (T-\la_1 E_r)(T-\la_0 E_r)^{-1} \wha E_r^*
&= \wha E_r\big[ I_{L^2_A(\bbR)^{2m}}-(\la_1-\la_0)(T-\la_0 E_r)^{-1}\big] \wha E_r^* \no
\\
&= I_{L^2_W(\bbR)^r}-(\la_1-\la_0)R(\la_0).     \lb{2.36}
\end{align}
Using \eqref{2.32}, \eqref{2.33}, and $\wha E_r^*\wha E_r T = E_r T = T$ on $\dom(T)$, one computes
\begin{align}
&(\la_0-\la_1)S\big[R(\la_0)-(\la_1-\la_0)^{-1}I_{L^2_W(\bbR)^r}\big]
= S\big[I_{L^2_W(\bbR)^r}-(\la_1-\la_0)R(\la_0)\big] \no
\\
&\quad = \wha E_r (T-\la_0 E_r)(T-\la_1 E_r)^{-1} \wha E_r^*\wha E_r (T-\la_1 E_r)(T-\la_0 E_r)^{-1} \wha E_r^*
\\
&\quad = \wha E_r (T-\la_0 E_r)(T-\la_1 E_r)^{-1} (T-\la_1 E_r)(T-\la_0 E_r)^{-1} \wha E_r^* = \wha E_r\wha E_r^* = I_{L^2_W(\bbR)^r},    \no
\end{align}
and similarly
\begin{align}
&\big[R(\la_0)-(\la_1-\la_0)^{-1}I_{L^2_W(\bbR)^r}\big](\la_0-\la_1)S = I_{L^2_W(\bbR)^r}.
\end{align}
Thus, $R(\la_0)-(\la_1-\la_0)^{-1}I_{L^2_W(\bbR)^r}$ is invertible and hence $(\la_1-\la_0)^{-1}\notin\si(R(\la_0))$.

Conversely, suppose $(\la_1-\la_0)^{-1}\in\bbC\bs\si(R(\la_0))$. As before one has \eqref{2.32}, \eqref{2.33} for $j=0$ and \eqref{2.36}. Let $S$ denote the inverse of $R(\la_0)-(\la_1-\la_0)^{-1}I_{L^2_W(\bbR)^r}$, then
\begin{align}
I_{L^2_W(\bbR)^r} &= S\big[R(\la_0)-(\la_1-\la_0)^{-1}I_{L^2_W(\bbR)^r}\big]  \no
\\
&= (\la_0-\la_1)^{-1} S \wha E_r (T-\la_1 E_r)(T-\la_0 E_r)^{-1} \wha E_r^*.
\end{align}
Applying both sides to the function $\wha E_r (T-\la_0 E_r)F \in L^2_W(\bbR)^r$, $F\in\dom(T)$, and recalling that $\wha E_r^*\wha E_r T = E_r T = T$ on $\dom(T)$, then yield via \eqref{2.32},
\begin{align}
(\la_0-\la_1)^{-1} S \wha E_r (T-\la_1 E_r)F = \wha E_r (T-\la_0 E_r)F, \quad F\in\dom(T).
\end{align}
Applying $(T-\la_0)^{-1}\wha E^*$ to both sides then similarly yields
\begin{align}
(\la_0-\la_1)^{-1} (T-\la_0)^{-1}\wha E^* S \wha E_r (T-\la_1 E_r)F = F, \quad F\in\dom(T).
\end{align}
An analogous computation also yields
\begin{align}
(T-\la_1 E_r) (\la_0-\la_1)^{-1} (T-\la_0)^{-1}\wha E^* S \wha E_r F = F, \quad F\in E_r L^2_A(\bbR)^{2m}.
\end{align}
Since $(\la_0-\la_1)^{-1} (T-\la_0)^{-1}\wha E^* S \wha E_r$ is a bounded operator from $E_r L^2_A(\bbR)^{2m}$ to $\dom(T)$, $T-\la_1 E_r$ has a bounded inverse, and hence, $\la_1\notin\si(T)$.
\end{proof}

Returning to the Hamiltonian system \eqref{2.10}, one recalls (cf., e.g.,
\cite{HS81}) that for $z \in \bbC \bs \bbR$ and a fixed reference point
$x_0 \in (a,b)$ {\it Weyl--Titchmarsh solutions}
$\Psi_{-, \al}(z,\dott,x_0) \in \bbC^{2m \times m}$ if $- \infty < a$,
$\Psi_{+, \be}(z,\dott,x_0) \in \bbC^{2m \times m}$ if $b < \infty$, and
$\Psi_{+}(z,\dott,x_0) \in \bbC^{2m \times m}$ if $b = \infty$,
$\Psi_{-}(z,\dott,x_0) \in \bbC^{2m \times m}$ if $a = - \infty$
of \eqref{2.10} are defined as follows: If $- \infty < a$, $\Psi_{-, \al}(z,\dott,x_0)$ satisfies the $\al$-boundary condition at $x=a$,
\begin{equation}
\al^* J \Psi_{-,\al}(z, a,x_0) = 0.
\end{equation}
Similarly, if $b < \infty$, $\Psi_{+, \be}(z,\dott,x_0)$ satisfies the
$\be$-boundary condition at $x=b$,
\begin{equation}
\be^* J \Psi_{+,\be}(z, b,x_0) = 0.
\end{equation}
If $b = \infty$, $\Psi_{+}(z,\dott,x_0)$ satisfies for all $c \in (a, \infty)$,
\begin{equation}
\Psi_{+}(z,\dott,x_0) \in L^2_A((c,\infty))^{2m},
\end{equation}
and if $a = - \infty$, $\Psi_{-}(z,\dott,x_0)$ satisfies for all $c \in (- \infty, b)$,
\begin{equation}
\Psi_{-}(z,\dott,x_0) \in L^2_A((- \infty,c))^{2m}.
\end{equation}

The actual choice of reference point is immaterial for the discussion in the remainder of this paper, but since Weyl--Titchmarsh solutions and matrices explicitly depend on it, we decided to indicate that explicitly in our choice of notation.

The normalization of each Weyl--Titchmarsh solution is fixed by
\begin{align}
\Psi_{-,\al}(z,\dott,x_0) &= U(z,\dott,x_0) (I_m \;\; M_{-,\al}(z,x_0))^\top,   \lb{2.47} \\
\Psi_{+, \be}(z,\dott,x_0) &= U(z,\dott,x_0) (I_m \;\; M_{+,\be}(z,x_0))^\top,    \lb{2.48} \\
\Psi_{\pm}(z,\dott,x_0) &= U(z,\dott,x_0) (I_m \;\; M_{\pm}(z,x_0))^\top,    \lb{2.49}
\end{align}
where $U(z,\dott,x_0)$ is a fundamental system of solutions of \eqref{2.10}
normalized by
\begin{equation}
U(z,x_0,x_0) = I_{2m},    \lb{2.50}
\end{equation}
and $M_{-,\al}(\dott,x_0)$, $M_{+,\be}(\dott,x_0)$, and
$M_{\pm}(\dott,x_0)$ are the
{\it Weyl--Titchmarsh functions}, in particular, $-M_{-,\al}(\dott,x_0)$,
$M_{+,\be}(\dott,x_0)$, and $\pm M_{\pm}(\dott,x_0)$  are all
$m \times m$ Nevanlinna-Herglotz matrices of full rank (i.e., analytic on the open upper half-plane, $\bbC_+$, with positive definite imaginary part on $\bbC_+$).
The Weyl--Titchmarsh solutions as well as functions extend analytically to all
$\bbC \bs \si(T_{a,b})$ (resp.,
$\bbC \bs \si(T_a)$ or $\bbC \bs \si(T)$).
In particular, if $- \infty < a < b < \infty$, $M_{-,\al}(\dott,x_0)$ and
$M_{+,\be}(\dott,x_0)$ are meromorphic.

In the following, we will call a solution $\Psi (\la,\dott) \in \bbC^{2m \times m}$, $\la \in \bbR$, of \eqref{2.10} {\it nondegenerate}, if for some (and hence for all) $x \in (a,b)$,
\begin{equation}
\Psi(\la,x)^* \Psi(\la, x) > 0, \quad \Psi(\la,x)^* J \Psi(\la,x) = 0.
\lb{2.51}
\end{equation}
(The first condition extends to all $x$ due to unique solvability of \eqref{2.10} and the second due to the Wronskian relation \eqref{2.20} in the special case $z_1=z_2=\la \in \bbR$, $D_j =0$, $j=1,2$.) Clearly, $\Psi_{-, \al}(\la,\dott,x_0)$,
$\Psi_{+, \be}(\la,\dott,x_0)$, $\Psi_{\pm}(\la,\dott,x_0)$ are all nondegenerate upon checking the conditions \eqref{2.51} at $x=x_0$.

We conclude this introductory section with two auxiliary results on nondegenerate solutions:

\begin{lemma} \lb{l2.4}
Assume Hypothesis \ref{h2.2}. Suppose $\Psi \in AC_\loc((a,b))^{2m \times m}$ is a nondegenerate solution of $\tau \Psi = \la E_r \Psi$, $\la\in\bbR$. Then there exist
$\te, \rho \in \bbC^{m \times m}$ satisfying
\begin{equation}
\te(x) = \te(x)^*, \quad \rho(x)^* \rho(x) > 0, \quad x \in (a,b),   \lb{2.52}
\end{equation}
such that
\begin{equation}
\Psi(x) = \big(\sin(\te(x)) \;\; \cos(\te(x))\big)^{\top} \rho(x), \quad x \in (a,b).  \lb{2.53}
\end{equation}
\end{lemma}
\begin{proof}
Since $\Psi(x)$ is nondegenerate, one has $\Psi(x)^* \Psi(x) > 0$ and $\Psi(x)^* J \Psi(x) = 0$.
Introducing
\begin{equation}
V_{\pm}(x) = (\pm I_m \;\; i I_m) \Psi(x) \in \bbC^{m \times m}, \quad x \in (a,b),
\end{equation}
one then infers
\begin{equation}
V_{\pm}(x)^* V_{\pm}(x) = \Psi(x)^* (I_{2m} \mp i J)\Psi(x) = \Psi(x)^* \Psi(x) >0.
\end{equation}
In particular, $\|V_+(x) h\|_{\bbC^m} = \|V_-(x) h\|_{\bbC^m} > 0$ for all $h \in \bbC^m \bs \{0\}$ therefore $V_{\pm}(x)$ are invertible and hence $U(x)=V_-(x)V_+(x)^{-1} \in \bbC^{m \times m}$ is unitary. Let $\te(x) = \te(x)^* \in \bbC^{m \times m}$ be such that,
\begin{equation}
U(x) = e^{2i \te(x)}, \quad x \in (a,b).
\end{equation}
Then
$e^{i \te} V_+ = e^{- i \te} V_-$ and hence
$(e^{i \te} \;\; ie^{i \te}) \Psi = (-e^{-i\te} \;\; i e^{-i\te}) \Psi$
implying
\begin{equation}
(\cos(\te) \;\; -\sin(\te)) \Psi = 0.
\end{equation}
Defining
\begin{equation}
\rho (x) = \big(\sin(\te(x)) \;\; \cos(\te(x))\big) \Psi(x), \quad x \in (a,b),
\end{equation}
and denoting
\begin{equation}
\Psi = (\Psi_1 \;\; \Psi_2)^{\top},
\end{equation}
one infers $\cos(\te) \Psi_1 = \sin(\te) \Psi_2$ and hence,
\begin{equation}
\sin(\te) \rho = [\sin(\te)]^2 \Psi_1 + \sin(\te) \cos(\te) \Psi_2
= [\sin(\te)]^2 \Psi_1 + [\cos(\te)]^2 \Psi_1 = \Psi_1
\end{equation}
and
\begin{equation}
\cos(\te) \rho = \cos(\te) \sin(\te) \Psi_1 + [\cos(\te)]^2 \Psi_2
= [\sin(\te)]^2 \Psi_2 + [\cos(\te)]^2 \Psi_2 = \Psi_2.
\end{equation}
Here we used the fact that $\theta = \theta^*$, and hence,
$\sin(\theta)$ and $\cos(\theta)$ commute. Thus, $\Psi = (\sin(\te) \;\; \cos(\te))^{\top} \rho$, implying
\begin{equation}
0 < \Psi^* \Psi = \rho^* \big(\sin(\te) \;\; \cos(\te)\big)
\big(\sin(\te) \;\; \cos(\te)\big)^\top \rho = \rho^* \rho.
\end{equation}
\end{proof}

\begin{lemma}\lb{l2.5}
Assume Hypothesis \ref{h2.2}. Suppose $\Psi \in AC_\loc((a,b))^{2m \times m}$ is a nondegenerate solution of $\tau \Psi = \la E_r \Psi$, $\la\in\bbR$. Then $\Psi$ satisfies the following analog of \eqref{2.19}
\begin{align}\lb{2.70}
\Psi [\Psi^*\Psi]^{-1}\Psi^* - J\Psi [\Psi^*\Psi]^{-1}\Psi^*J = I_{2m}.
\end{align}
\end{lemma}
\begin{proof}
Let $A=\Psi^*\Psi$. Then by \eqref{2.51},
\begin{align}
\big(\Psi \;\; -J\Psi\big)^*\big(\Psi A^{-1} \;\; J\Psi A^{-1}\big) = I_{2m}.
\end{align}
Since $\big(\Psi \;\; -J\Psi\big)$ and $\big(\Psi A^{-1} \;\; J\Psi A^{-1}\big)$ are finite-dimensional square matrices, we also have
\begin{align}
\big(\Psi A^{-1} \;\; J\Psi A^{-1}\big)\big(\Psi \;\; -J\Psi\big)^* = I_{2m},
\end{align}
which is \eqref{2.70}.
\end{proof}

\section{The Half-Line Case $[a,\infty)$} \label{s3}

In this section we consider the half-line case $[a,\infty)$, $-\infty < a < b = \infty$.
The compact interval case $[a,b]$ is analogous upon consistently replacing
$T_a$ by $T_{a,b}$ below. For this reason we keep the notation $b$ even though $b=\infty$ in this section.

\begin{hypothesis} \lb{h3.1}
Fix $x_0,\la_0,\la_1\in\bbR$, $\la_0<\la_1$, and assume the Weyl--Titchmarsh solutions $\Psi_+(\la_0,\dott,x_0)$ and $\Psi_{-,\al}(\la_1,\dott,x_0)$ are well-defined.
In addition, for $c\in(a,b)$ define
\begin{equation}
\ga := \ga(\la_0,c,x_0)
= \big(\sin(\te_+(\la_0,c,x_0)) \;\; \cos(\te_+(\la_0,c,x_0))\big)^{\top} \in \bbC^{2m\times m},  \lb{3.1}
\end{equation}
$($satisfying \eqref{2.17}, equivalently, \eqref{2.18}, \eqref{2.19}$)$, where $\te_+(\la_0,\dott,x_0)$ is the Pr\"ufer angle of the Weyl--Titchmarsh solution $\Psi_+(\la_0,\dott,x_0)$ introduced in Lemma~\ref{l2.4}.
\end{hypothesis}

For the purpose of restricting \eqref{2.10} to the interval $(a,c)$ we now introduce the orthogonal projection operator in $L^2_A((a,b))^{2m}$,
\begin{equation}
(P_c f)(x) := \begin{cases} f(x), & x \in (a,c), \\ 0, & x \in [c,b), \end{cases}
\quad f \in L^2_A((a,b))^{2m}.
\end{equation}
With a slight abuse of notation, we will denote the analogous projection operator in the space $L^2_W((a,b))^r$ by the same symbol $P_c$.
The operator associated with \eqref{2.10} restricted to the interval $(a,c)$ will be denoted by $T_{a,c}$ with $\al$ and $\ga$ (cf., \eqref{3.1}) defining the boundary conditions at $x=a$ and $x=c$, respectively. Equivalently, by \eqref{3.1}, $G \in \dom(T_{a,c})$ satisfies the boundary condition at $x=c$ of the form
\begin{equation}
\Psi_+(\la_0,c,x_0)^* J G(c) = 0.   \lb{3.4}
\end{equation}
Also the boundary condition $\al^* J G(a) = 0$ at $x=a$ can be restated in terms of the Weyl--Titchmarsh solution $\Psi_{-,\al}(\la,\dott,x_0)$ for any $\la \in \bbR$ for which $\Psi_{-,\al}(\la,\dott,x_0)$ is well-defined. Let $\rho_-(\la,a,x_0) = \al^* \Psi_{-,\al}(\la,a,x_0)$. Then, using \eqref{2.19} and $\al^* J \Psi_{-,\al}(\la,a,x_0) = 0$, one obtains
\begin{equation}
\al \rho_-(\la,a,x_0) = \al \al^* \Psi_{-,\al}(\la,a,x_0)
= (I_{2m} + J \al \al^* J) \Psi_{-,\al}(\la,a,x_0) = \Psi_{-,\al}(\la,a,x_0).
\end{equation}
Since $\Psi_{-,\al}(\la,\dott,x_0)$ is nondegenerate, it follows that
$\rho_-(\la,a,x_0)$ is invertible and hence $\al = \Psi_{-,\al}(\la,a,x_0) \rho_-(\la,a,x_0)^{-1}$. Thus, $G \in \dom(T_{a,c})$ satisfies the boundary condition at $x=a$ of the form
\begin{equation} \lb{3.5}
\Psi_{-,\al}(\la,a,x_0)^* J G(a) = 0.
\end{equation}

Next, we recall the structure of the resolvent and Green's function of $T_a$,
\begin{align}
\begin{split}
\big((T_a - z E_r)^{-1} G\big)(x) =
\int_a^b dx' \, K_a(z,x,x') A(x') G(x'),&
\\
z\in\bb C\bs\si(T_a),\; G \in E_r L^2_A((a,b))^{2m},&    \lb{3.6}
\end{split}
\end{align}
where
\begin{align}
K_a(z,x,x') =
\begin{cases}
\Psi_{-,\al}(z,x,x_0) \cW(z)^{-1}\Psi_+(\ol{z},x',x_0)^*, & x < x', \\
\Psi_+(z,x,x_0) \cW(z)^{-1} \Psi_{-,\al}(\ol{z},x',x_0)^*, & x > x',
\end{cases}    \lb{3.7}
\end{align}
and $\cW(z)$ is the ($x$-independent) Wronskian
\begin{align}
\begin{split}
\cW(z) &= -\Psi_+(\ol{z},\dott,x_0)^* J \Psi_{-,\al}(z,\dott,x_0) \\
&= \Psi_{-,\al}(\ol{z},\dott,x_0)^* J \Psi_+(z,\dott,x_0)
= M_{-,\al}(z,x_0) - M_+(z,x_0).   \lb{3.8}
\end{split}
\end{align}
The resolvents of $T_{a,b}$ and $T$ are given by analogous formulas. To see that the right-hand side of \eqref{3.6} is the inverse of $T_a-z E_r$, one first notes that
\begin{align}
\Psi_{-,\al}(\ol{z},x,x_0)^* J \Psi_{-,\al}(z,x,x_0)
&= (I_m \;\; M_{-,\al}(\ol{z},x_0)^*) J (I_m \;\; M_{-,\al}(z,x_0))^\top,   \no
\\
&= (I_m \;\; M_{-,\al}(z,x_0)) J (I_m \;\; M_{-,\al}(z,x_0))^\top = 0_m,
\\
\Psi_+(\ol{z},x,x_0)^* J \Psi_+(z,x,x_0) &=
(I_m \;\; M_+(z,x_0)) J (I_m \;\; M_+(z,x_0))^\top = 0_m,   \no
\end{align}
and hence, by \eqref{3.8},
\begin{align}
\binom{-\Psi_+(\ol{z},x,x_0)^*}{\Psi_{-,\al}(\ol{z},x,x_0)^*} J \big(\Psi_{-,\al}(z,x,x_0)\cW(z)^{-1} \;\; \Psi_+(z,x,x_0)\cW(z)^{-1}\big) = I_{2m}.
\end{align}
Since for a square matrix the left inverse equals the right inverse, it follows that
\begin{align}
&\big(\Psi_{-,\al}(z,x,x_0)\cW(z)^{-1} \;\; \Psi_+(z,x,x_0)\cW(z)^{-1}\big) \binom{-\Psi_+(\ol{z},x,x_0)^*}{\Psi_{-,\al}(\ol{z},x,x_0)^*}       \lb{3.9}
\\
&\quad = \Psi_+(z,x,x_0)\cW(z)^{-1}\Psi_{-,\al}(\ol{z},x,x_0)^*
- \Psi_{-,\al}(z,x,x_0)\cW(z)^{-1}\Psi_+(\ol{z},x,x_0)^* = J^{-1}.  \no
\end{align}
Then, using $(\tau-zE_r)\Psi_+=0$ and $(\tau-zE_r)\Psi_{-,\al}=0$, one verifies
\begin{align}
&(T_a-z E_r) \int_a^b dx' \, K_a(z,x,x') A(x')G(x')    \no
\\
&\quad = C(x)J\big[\Psi_+(z,x,x_0)\cW(z)^{-1}\Psi_{-,\al}(\ol{z},x,x_0)^*
\\
&\hspace{19mm} -
\Psi_{-,\al}(z,x,x_0)\cW(z)^{-1}\Psi_+(\ol{z},x,x_0)^*\big]A(x)G(x)   \no
\\
&\quad = C(x)JJ^{-1}A(x)G(x) = E_r G(x) = G(x), \quad G \in E_r L^2_A((a,b))^{2m}. \no
\end{align}

\begin{lemma} \lb{l3.1}
Assume Hypotheses \ref{h2.2}, \ref{h3.1}, and $\la_0 \in \bbR \bs \si(T_a)$. Then $\la_0\notin\si(T_{a,c})$ and
\begin{align} \lb{3.10}
\big((T_a - \la_0 E_r)^{-1} P_c G\big)\big|_{(a,c)} =
(T_{a,c} - \la_0 E_r)^{-1}(G\vert_{(a,c)}),
\quad G \in E_r L^2_A((a,b))^{2m}.
\end{align}
\end{lemma}
\begin{proof}
Since $\Psi_+(\la_0,\dott,x_0)$ is a nondegenerate solution, it satisfies
\begin{equation}
\Psi_+(\la_0,c,x_0)^* J \Psi_+(\la_0,c,x_0) = 0,
\end{equation}
and hence, by \eqref{3.4}, $\Psi_+(\la_0,\dott,x_0)$ satisfies the boundary condition at $x=c$. Thus, $\Psi_+(\la_0,\dott,x_0)$ is also the Weyl--Titchmarsh solutions for $T_{a,c}$ and hence $(T_{a,c}-\la_0E_r)^{-1}$ is given by formulas completely analogous to \eqref{3.6}--\eqref{3.8} (employing the same
$\Psi_{-,\al}(\la_0,\dott,x_0)$, $\Psi_+(\la_0,\dott,x_0)$). This yields relation \eqref{3.10}.
\end{proof}

Introducing in $L^2_W((a,c))^r$ the restricted resolvent of $T_{a,c}$
(cf.\ \eqref{2.31}),
\begin{equation}
R_{a,c}(\la_0) := \wha E_r (T_{a,c} - \la_0 E_r)^{-1}
{\wha E_r}^*, \quad \la_0 \in \bbR \bs \si(T_a),   \lb{3.13}
\end{equation}
Lemma \ref{l3.1} can be rewritten as follows:

\begin{corollary} \lb{c3.2}
Assume Hypotheses \ref{h2.2}, \ref{h3.1}, and $\la_0 \in \bbR \bs \si(T_a)$. Then $\la_0\notin\si(T_{a,c})$ and
\begin{equation}
P_c R_a(\la_0) P_c = R_{a,c}(\la_0)\oplus0.
\end{equation}
\end{corollary}

In the following, for a linear operator $S$ in an appropriate linear space we introduce the notation,
\begin{equation}
N(S) := \dim(\ker(S))
\end{equation}
($N(S)$ is also called the nullity, ${\rm nul} (S)$, of $S$), in addition, we will employ the
symbol $N((\la_0, \la_1); S)$ to denote the sum of geometric multiplicities of
all eigenvalues of $S$ in the interval $(\la_0, \la_1)$.

\begin{theorem} \lb{t3.4}
Assume Hypotheses \ref{h2.2}, \ref{h3.1}, and $\la_0 \in \bbR \bs \si(T_a)$. Then,
\begin{align}
N(T_{a,c} - \la_1 E_r)
&= N\big(R_{a,c}(\la_0) - (\la_1 - \la_0)^{-1}I_{L^2_W((a,c))^{r}}\big) \no \\
&= N\big(\Psi_+(\la_0,c,x_0)^* J \Psi_{-,\al}(\la_1,c,x_0)\big).    \lb{3.25}
\end{align}
\end{theorem}
\begin{proof}
Applying Lemma~\ref{l2.3} to $T_{a,c}$, $R_{a,c}(\la_0)$ yields the first equality in \eqref{3.25}.

Next, suppose $N\big(\Psi_+(\la_0,c,x_0)^* J \Psi_{-,\al}(\la_1,c,x_0)\big)=n$ and let $\{v_j\}_{j=1}^n$ be a basis of the kernel of
$\Psi_+(\la_0,c,x_0)^* J \Psi_{-,\al}(\la_1,c,x_0)$. Then $F_j(x)=
\Psi_{-,\al}(\la_1,x,x_0) v_j$, $j=1,\dots,n$, are linearly independent elements of $\ker(T_{a,c} - \la_1 E_r)$ and hence,
\begin{equation}
N(T_{a,c} - \la_1 E_r)\geq N\big(\Psi_+(\la_0,c,x_0)^* J \Psi_{-,\al}(\la_1,c,x_0)\big).  \lb{3.27}
\end{equation}

Conversely, suppose $N(T_{a,c} - \la_1 E_r)=n$ and let $\{F_j(x)\}_{j=1}^n$ be a basis of the kernel of $T_{a,c} - \la_1 E_r$. Then the functions $F_j(x)$ satisfy the boundary condition \eqref{3.5} at $x=a$,
\begin{align}
\Psi_{-,\al}(\la_1,a,x_0)^*J F_j(a) = 0, \quad j=1,\dots,n.   \lb{3.28}
\end{align}
Applying $J\Psi_{-,\al}(\la_1,a,x_0) [\Psi_{-,\al}(\la_1,a,x_0)^*\Psi_{-,\al}(\la_1,a,x_0)]^{-1}$ to \eqref{3.28} and employing relation \eqref{2.70} with
$\Psi = \Psi_{-,\al}(\la_1,a,x_0)$ then yield
\begin{align}
F_j(a) = \Psi_{-,\al}(\la_1,a,x_0) v_j, \quad j=1,\dots,n,
\end{align}
where
\begin{align}
v_j= [\Psi_{-,\al}(\la_1,a,x_0)^*\Psi_{-,\al}(\la_1,a,x_0)]^{-1} \Psi_{-,\al}(\la_1,a,x_0)^* F_j(a), \quad j=1,\dots,n,
\end{align}
are linearly independent vectors in $\bbC^{m}$.
Since solutions of \eqref{2.10} are uniquely determined by their initial conditions, it follows that
\begin{align}
F_j(x)=\Psi_{-,\al}(\la_1,x,x_0) v_j, \quad x\in[a,c], \; j=1,\dots,n.
\end{align}
Moreover, since $F_j$ also satisfy the boundary condition \eqref{3.4} at
$x=c$, one concludes
\begin{align}
\Psi_+(\la_0,c,x_0)^* J \Psi_{-,\al}(\la_1,c,x_0) v_j = 0, \quad j=1,\dots,n,
\end{align}
that is,
\begin{equation}
N\big(\Psi_+(\la_0,c,x_0)^* J \Psi_{-,\al}(\la_1,c,x_0)\big)
\geq N(T_{a,c} - \la_1 E_r).   \lb{3.29}
\end{equation}
Inequalities \eqref{3.27} and \eqref{3.29} imply the second equality in \eqref{3.25}, concluding the proof.
\end{proof}

\begin{lemma} \lb{l3.5}
Assume Hypotheses \ref{h2.2}, \ref{h3.1}, and $\la_0 \in \bbR \bs \si(T_a)$. Then the eigenvalues of $R_{a,c}(\la_0)$ and $T_{a,c}$ are monotone continuous functions of $c \in (a,b)$.
\end{lemma}
\begin{proof}
Note that similarly to Corollary~\ref{c3.2} one has
\begin{equation}
R_{a,c}(\la_0) \oplus 0 = P_c R_{a,d}(\la_0) P_c,  \lb{3.35}
\end{equation}
where $R_{a,d}(\la_0)$ in $L^2_W((a,d))^r$ is defined as $R_{a,c}(\la_0)$ with $c$ replaced by $d \in (a,b)$, $c < d$. The projection operator $P_c$ is now considered in $L^2_W((a,d))^r$. It is continuous with respect to $c$ in the strong operator topology. Since $R_{a,d}(\la_0)$ has a square integrable integral kernel, the operator $R_{a,d}(\la_0)$ is Hilbert--Schmidt (and hence compact) in $L^2_W((a,d))^r$. Thus, $P_cR_{a,d}(\la_0)P_c$ is continuous with respect to $c$ in the uniform operator topology. Consequently, by \eqref{3.35}, the eigenvalues of $R_{a,c}(\la_0)$, and by Lemma~\ref{l2.3} those of $T_{a,c}$, are continuous with respect to $c$ (see, e.g.,
\cite[Theorem~VIII.23]{RS80}, \cite[Theorem~9.5]{We80}).

Since for every $F \in L^2_W((a,c))^r$, its zero extension to $(a,d)$,
\begin{equation}
\wti F(x) = \begin{cases} F(x), & x \in (a,c), \\ 0, & x \in [c,d), \end{cases}
\end{equation}
satisfies
\begin{equation}
(F, R_{a,c}(\la_0) F)_{L^2_W((a,c))^r} =
\big(\wti F, R_{a,d}(\la_0) \wti F\big)_{L^2_W((a,d))^r},
\end{equation}
it follows from the min-max principle that for every $\mu>0$,
\begin{align}
\dim\bigl(\ran\bigl(P\big((-\infty,-\mu);R_{a,c}(\la_0)\big)\big)\big) &\leq
\dim\bigl(\ran\bigl(P\big((-\infty,-\mu);R_{a,d}(\la_0)\big)\big)\big),
\\
\dim\bigl(\ran\bigl(P\big((\mu,\infty);R_{a,c}(\la_0)\big)\big)\big) &\leq
\dim\bigl(\ran\bigl(P\big((\mu,\infty);R_{a,d}(\la_0)\big)\big)\big).
\end{align}
Thus, the eigenvalues of $R_{a,c}(\la_0)$ are monotone (negative ones are nonincreasing, positive ones are nondecreasing) as $c$ increases. Then, by Lemma~\ref{l2.3}, the eigenvalues of $T_{a,c}$ are monotone as well.
\end{proof}

One half of the principal result of this section is stated next:

\begin{theorem} \lb{t3.6}
Assume Hypotheses \ref{h2.2}, \ref{h3.1} and $\la_0, \la_1 \in \bbR \bs \si(T_a)$. Then,
\begin{equation}
N((\la_0,\la_1); T_a) \leq
\sum_{x \in (a,b)} N\big(\Psi_+(\la_0,x,x_0)^* J \Psi_{-,\al}(\la_1, x,x_0)\big).
\end{equation}
\end{theorem}
\begin{proof}
Let $\mu=\la_1-\la_0>0$. Then by Lemma~\ref{l2.3},
\begin{equation}
N((\la_0,\la_1); T_a) \leq
\dim\bigl(\ran\bigl(P\big((\mu^{-1},\infty); R_a(\la_0)\big)\big)\big). \lb{3.40}
\end{equation}
Since $P_c\underset{c \uparrow b}{\longrightarrow} I_{L^2_W((a,b))^r}$ in the strong operator topology, one has
\begin{equation}
R_{a,c}(\la_0) \oplus 0 = P_c R_a(\la_0) P_c
\underset{c \uparrow b}{\longrightarrow} R_a(\la_0)
\end{equation}
in the strong operator topology in $L^2_W((a,b))^r$, and hence (cf.\ \cite[Lemma~5.2]{GST96a}),
\begin{align}
\dim\bigl(\ran\bigl(P\big((\mu^{-1},\infty); R_a(\la_0)\big)\big)\big) &\leq
\liminf_{c \uparrow b} \,
\dim\bigl(\ran\bigl(P\big((\mu^{-1},\infty); P_cR_a(\la_0)P_c\big)\big)\big) \no \\
&= \liminf_{c \uparrow b} \,
\dim\bigl(\ran\bigl(P\big((\mu^{-1},\infty); R_{a,c}(\la_0)\big)\big)\big). \lb{3.42}
\end{align}
Since $P_c\underset{c \downarrow a}{\longrightarrow} 0$ in the strong operator topology one concludes as in the proof of Lemma \ref{l3.5} that $R_{a,c}(\la_0) \underset{c \downarrow a}{\longrightarrow} 0$ in the norm operator topology. Then since the positive eigenvalues of $R_{a,c}(\la_0)$ are nondecreasing continuous functions of $c$, one concludes that
\begin{equation}
\dim\bigl(\ran\bigl(P\big((\mu^{-1},\infty); R_{a,c}(\la_0)\big)\big)\big)
\leq \sum_{x \in (a,c)} N\big(R_{a,x}(\la_0) - \mu^{-1} I_{L^2_W((a,c))^{r}}\big).  \lb{3.43}
\end{equation}
Combining \eqref{3.43} with \eqref{3.40} and \eqref{3.42} then yields
\begin{align}
N((\la_0,\la_1); T_a) \leq \sum_{x \in (a,b)} N\big(R_{a,x}(\la_0) - \mu^{-1} I_{L^2_W((a,c))^{r}}\big).
\end{align}
Finally, an application of Theorem \ref{t3.4} completes the proof.
\end{proof}

Next we record an auxiliary result.

\begin{lemma} \lb{l3.7}
Assume Hypotheses \ref{h2.2}, \ref{h3.1}. If $c\in(a,b)$ is such that
\begin{equation}
N\big(\Psi_+(\la_0,c,x_0)^* J \Psi_{-,\al}(\la_1,c,x_0)\big) = n >0,
\end{equation}
then there exist $\{v^\pm_j\}_{1 \leq j \leq n} \subset \bbC^m$ so that
\begin{equation}
F_j(x) = \begin{cases}
\Psi_{-,\al}(\la_1,x,x_0) v^-_j, \\
\Psi_+(\la_0,x,x_0) v^+_j,
\end{cases} 1 \leq j \leq n,
\end{equation}
satisfy
\begin{equation}
F_j \in \dom(T_a) \bs \{0\} \, \text { and } \,
(T_a F_j)(x) = E_r F_j(x) \cdot
\begin{cases} \la_1, & x \in (a,c), \\
\la_0, & x \in (c,b), \end{cases} \quad 1 \leq j \leq n.    \lb{3.46}
\end{equation}
\end{lemma}
\begin{proof}
Let $\{v^-_j\}_{1 \leq j \leq n}$ be a basis of
$\ker\big(\Psi_+(\la_0,c,x_0)^* J \Psi_{-,\al}(\la_1,c,x_0)\big)$.
By Lemma \ref{l2.4} there exists
$\ga\in\bbC^{2m\times m}$ satisfying \eqref{2.17} (equiv., \eqref{2.18}, \eqref{2.19}) and invertible $\rho\in\bbC^{m\times m}$ such that $\Psi_+(\la_0,c,x_0) = \ga\rho$.
Defining
\begin{equation}
v^+_j = \rho^{-1} \ga^* \Psi_{-,\al}(\la_1,c,x_0)v^-_j, \quad 1 \leq j \leq n,
\end{equation}
then yields
$\Psi_+(\la_0,c,x_0) v^+_j = \ga \ga^* \Psi_{-,\al}(\la_1,c,x_0)v^-_j$, $1 \leq j \leq n$.
By construction,
\begin{align}
\ga^* J \Psi_{-,\al}(\la_1,c,x_0) v^-_j = (\rho^{-1})^*\Psi_+(\la_0,c,x_0)^* J \Psi_{-,\al}(\la_1,c,x_0) v^-_j = 0, \quad 1 \leq j \leq n,
\end{align}
and by \eqref{2.19}, $\ga \ga^* = J \ga \ga^* J + I_{2m}$, hence
\begin{align}
\Psi_+(\la_0,c,x_0) v^+_j &= (J \ga \ga^* J + I_{2m})
\Psi_{-,\al}(\la_1,c,x_0) v^-_j     \no \\
&= \Psi_{-,\al}(\la_1,c,x_0) v^-_j, \quad 1 \leq j \leq n.
\end{align}
Thus, $F_j$ are continuous at $x=c$, hence $F_j \in AC_\loc([a, b))^{2m}$
and $F_j \in \dom(T_a)$, $1 \leq j \leq n$. The second assertion in
\eqref{3.46} is clear.
\end{proof}

Employing the Wronskian identity one obtains the following orthogonality statement:

\begin{lemma} \lb{l3.8}
Assume Hypotheses \ref{h2.2}, \ref{h3.1}. Let $\{c_k\}_{k=1}^K \subset (a,b)$ be the points where
\begin{equation}
N\big(\Psi_+(\la_0,c_k,x_0)^* J \Psi_{-,\al}(\la_1,c_k,x_0)\big) =:n_k > 0.
\end{equation}
In addition, for each $c_k$, let $\{v^\pm_{k,j}\}_{1 \leq j \leq n_k}$ be as in Lemma \ref{l3.7}, and introduce
\begin{align}
u^-_{k,j} (x) &= \Psi_{-,\al}(\la_1,x,x_0) v^-_{k,j} \chi_{(a,c_k)}(x),   \no \\
u^+_{k,j} (x) &= \Psi_+(\la_0,x,x_0) v^+_{k,j} \chi_{(c_k,b)}(x),
\quad 1 \leq j \leq n_k, \; 1 \leq k \leq K.
\end{align}
Then,
\begin{equation}
(u^+_{k,j} , u^-_{\ell,i})_{L^2_A((a,b))^{2m}} = 0, \quad
1 \leq j \leq n_k, \; 1 \leq i \leq n_\ell, \; 1 \leq k, \ell \leq K. \lb{3.51}
\end{equation}
\end{lemma}
\begin{proof}
Using \eqref{2.20}, one computes,
\begin{align}
& (u^+_{k,j} , u^-_{\ell,i})_{L^2_A((a,b))^{2m}}
= \int_{c_k}^{c_{\ell}} dx \, (v^+_{k,j})^* \Psi_+(\la_0,x,x_0)^* A(x) \Psi_{-,\al}(\la_1,x,x_0) v^-_{\ell,i}   \no \\
& \quad = (\la_1-\la_0)^{-1} (v^+_{k,j})^* \Psi_+(\la_0,x,x_0)^* J
\Psi_{-,\al}(\la_1,x,x_0) v^-_{\ell,i} \big|_{c_k}^{c_{\ell}}.    \lb{3.52}
\end{align}
By construction,
\begin{equation}
v^-_{\ell,i} \in \ker\big(\Psi_+(\la_0,c_{\ell},x_0)^* J
\Psi_{-,\al}(\la_1,c_{\ell},x_0)\big), \quad 1 \leq i \leq n_\ell, \; 1\leq \ell \leq K,     \lb{3.53}
\end{equation}
and
\begin{equation}
\Psi_+(\la_0,c_k,x_0) v^+_{k,j} = \Psi_{-,\al}(\la_1,c_k) v^-_{k,j},
\quad 1 \leq j \leq n_k, \; 1 \leq k \leq K.
\end{equation}
Since $\Psi_{-,\al}(z,\dott,x_0)$ is nondegenerate, $\Psi_{-,\al}(z,\dott,x_0)^* J \Psi_{-,\al}(z,\dott,x_0) = 0$ yielding
\begin{align}
\begin{split}
& \Psi_{-,\al}(\la_1,c_k,x_0)^* J \Psi_+(\la_0,c_k,x_0) v^+_{k,j}    \\
& \quad =
\Psi_{-,\al}(\la_1,c_k,x_0)^* J \Psi_{-,\al}(\la_1,c_k,x_0) v^-_{k,j} = 0,
\quad 1 \leq j \leq n_k, \; 1 \leq k \leq K,
\end{split}
\end{align}
that is,
\begin{equation}
v^+_{k,j} \in \ker\big(\Psi_-(\la_1,c_k,x_0)^* J \Psi_+(\la_0,c_k,x_0)\big),
\quad 1 \leq j \leq n_k, \; 1 \leq k \leq K.    \lb{3.54}
\end{equation}
Thus, \eqref{3.51} follows from \eqref{3.52}, \eqref{3.53}, and \eqref{3.54}.
\end{proof}

This leads to the second half of the principal result of this section:

\begin{theorem} \lb{t3.9}
Assume Hypotheses \ref{h2.2}, \ref{h3.1}, $\la_0, \la_1 \in \bbR \bs \si(T_a)$, and $(\la_0,\la_1)\cap\si_\ess(T_a)=\emptyset$. Then,
\begin{equation}
N((\la_0,\la_1); T_a) \geq
\sum_{x \in (a,b)} N\big(\Psi_+(\la_0,x,x_0)^* J \Psi_{-,\al}(\la_1, x,x_0)\big).
\end{equation}
\end{theorem}
\begin{proof}
Let $\mu=\la_1-\la_0$. Since $\la_1\notin\si(T_a)$ it suffices to prove, by Lemma \ref{l2.3}, that
\begin{equation}
\dim\bigl(\ran\bigl(P\big([\mu^{-1},\infty); R_a(\la_0)\big)\big)\big) \geq
\sum_{x \in (a,b)} N\big(\Psi_+(\la_0,x,x_0)^* J \Psi_{-,\al}(\la_1, x,x_0)\big).
\lb{3.56}
\end{equation}
By the min-max principle it suffices to establish the existence of a subspace
$\wha \cS$ of $L^2_W((a,b))^r$ whose dimension equals the right-hand side of \eqref{3.56}, such that
\begin{equation}
(f,R_a(\la_0) f)_{L^2_W((a,b))^r} \geq \mu^{-1}
(f, f)_{L^2_W((a,b))^r},  \quad f \in \wha \cS.    \lb{3.57}
\end{equation}
To this end, let $\{c_k\}_{k=1}^K \subset (a,b)$ and $u^\pm_{k,j}$, $1 \leq j \leq n_k$, $1 \leq k \leq K$, be as in Lemma \ref{l3.8} and introduce
\begin{equation}
\cS = \ol{{\rm lin.span} \{u^-_{k,j} \,|\, 1 \leq j \leq n_k, \; 1 \leq k \leq K\}}, \quad
\wha \cS = \wha E_r \cS.
\end{equation}
Since the functions $u^-_{k,j}$, $1 \leq j \leq n_k$, $1 \leq k \leq K$, are linearly independent and
\begin{equation}
(f,g)_{L^2_A((a,b))^{2m}} = \big(\wha E_r f, \wha E_r g\big)_{L^2_W((a,b))^r},
\quad f,g \in L^2_A((a,b))^{2m},
\end{equation}
one concludes that
\begin{equation}
\dim\bigl(\wha \cS\,\big) = \dim(\cS) =
\sum_{x \in (a,b)} N\big(\Psi_+(\la_0,x,x_0)^* J \Psi_{-,\al}(\la_1, x,x_0)\big).
\end{equation}
Next, since $(T - \la_0 E_r) (u^+_{k,j} - u^-_{k,j}) = \mu E_r u^-_{k,j}$, one obtains upon introducing
\begin{equation}
\wha u^\pm_{k,j} = \wha E_r u^\pm_{k,j},
\quad 1 \leq j \leq n_k, \; 1 \leq k \leq K,
\end{equation}
that
\begin{equation}
R_a(\la_0) \wha u^-_{k,j} = \mu^{-1} \big(\wha u^+_{k,j} - \wha u^-_{k,j}\big),
\quad 1 \leq j \leq n_k, \; 1 \leq k \leq K.
\end{equation}
By linearity, also
\begin{equation}
R_a(\la_0) \wha f_- = \mu^{-1} \big(\wha f_+ + \wha f_-\big),
\quad \wha f_{\pm} = \sum_{k=1}^K \sum_{j=1}^{n_k} d_{k,j} \wha u^\pm_{k,j},
\end{equation}
with $d_{k,j} \in \bbC$, $1 \leq j \leq n_k$, $1 \leq k \leq K$. By Lemma \ref{l3.8},
$\wha f_+ \, \bot \, \wha f_-$ in $L^2_W((a,b))^r$ since
\begin{equation}
\big(\wha f_+ , \wha f_- \big)_{L^2_W((a,b))^r}
= \sum_{k=1}^K \sum_{j=1}^{n_k} \sum_{\ell=1}^K \sum_{i=1}^{n_{\ell}}
d_{k,j} \ol{d_{\ell,i}} (u^+_{k,j} , u^-_{\ell,i})_{L^2_A((a,b))^{2m}}
= 0.
\end{equation}
Thus,
\begin{equation}
\big(\wha f_-, R_a(\la_0) \wha f_-\big)_{L^2_W((a,b))^r}
= \mu^{-1} \big(\wha f_-, \wha f_-\big)_{L^2_W((a,b))^r},
\quad \wha f_- \in \wha \cS,
\end{equation}
implying \eqref{3.57}.
\end{proof}

Combining Theorems \ref{t3.6} and \ref{t3.9} thus yields the first new principal result of this paper:

\begin{theorem} \lb{t3.10}
Assume Hypotheses \ref{h2.2}, \ref{h3.1}, $\la_0, \la_1 \in \bbR \bs \si(T_a)$, and $(\la_0,\la_1)\cap\si_\ess(T_a)=\emptyset$. Then,
\begin{equation}
N((\la_0,\la_1); T_a) =
\sum_{x \in (a,b)} N\big(\Psi_+(\la_0,x,x_0)^* J \Psi_{-,\al}(\la_1, x,x_0)\big).
\end{equation}
\end{theorem}

We emphasize that the interval $(\la_0,\la_1)$ can lie in any essential spectral gap of $T_a$, not just below its essential spectrum as in standard approaches to oscillation theory. To the best of our knowledge, even the special scalar case $m=1$ appears to be new for general Hamiltonian systems.

\begin{remark} \lb{r3.11}
The case $\la_0, \la_1 \in \bbR \bs \si(T_a)$,
$\la_1 < \la_0$, $(\la_1,\la_0)\cap\si_\ess(T_a) = \emptyset$ is completely analogous. Similarly, one can interchange the roles of $\Psi_+$ and $\Psi_{-,\al}$ and employ
$\Psi_+(\la_1,c,x_0)$, $\Psi_{-,\al}(\la_0, c,x_0)$, etc.
\end{remark}

\section{The Real Line Case $(a,b) = \bbR$} \lb{s4}

In our final section we consider the full-line case $(a,b) = \bbR$, replacing the operator $T_a$ by $T$ (cf.\ \eqref{2.22}, \eqref{2.23}), still assuming Hypothesis \ref{h2.2} throughout. In addition, we make the following assumptions.

\begin{hypothesis} \lb{h4.1}
Fix $x_0,\la_0,\la_1\in\bbR$, $\la_0<\la_1$, and assume the Weyl--Titchmarsh solutions $\Psi_+(\la_0,\dott,x_0)$ and $\Psi_-(\la_1,\dott,x_0)$ are well-defined.
In addition, for $a\in\bbR$ define the boundary condition matrix
\begin{equation}
\al := \al(\la_1,a,x_0)
= (\sin(\te_-(\la_1,a,x_0)) \;\; \cos(\te_-(\la_1,a,x_0)))^{\top}
\in \bbC^{2m\times m},  \lb{4.1}
\end{equation}
$($satisfying \eqref{2.17}, equiv., \eqref{2.18}, \eqref{2.19}$)$, where $\te_-(\la_1,\dott,x_0)$ is the Pr\"ufer angle of the Weyl--Titchmarsh solution $\Psi_-(\la_1,\dott,x_0)$ introduced in Lemma~\ref{l2.4}. We continue denoting the half-line operator in \eqref{2.22} by $T_a$ with the boundary matrix $\al$ now defined as in \eqref{4.1}.
\end{hypothesis}

In the following we consider the orthogonal projection $P_a$ on $L^2_A(\bbR)^{2m}$ given by
\begin{equation}
(P_a f)(x) := \begin{cases} f(x), & x \in (a,\infty), \\ 0, & x \in (-\infty,a], \end{cases}
\quad f \in L^2_A(\bbR)^{2m}.
\end{equation}
With a slight abuse of notation, we will denote the analogous projection operator in $L^2_W(\bbR)^r$ by the same symbol $P_a$.

In analogy to Lemma \ref{l3.1}, one then obtains the following restriction result.

\begin{lemma} \lb{l4.1}
Assume Hypotheses \ref{h2.2}, \ref{h4.1}, and $\la_1 \in \bbR \bs \si(T)$. Then, $\la_1\notin\si(T_a)$ and
\begin{align}
\big((T - \la_1 E_r)^{-1} P_a G\big)\big|_{(a,\infty)} =
(T_a - \la_1 E_r)^{-1}(G\vert_{(a,\infty)}),
\quad G \in E_r L^2_A(\bbR)^{2m},   \lb{4.3}
\end{align}
and hence,
\begin{align} \lb{4.4}
P_a R(\la_1) P_a = R_a(\la_1)\oplus0.
\end{align}
\end{lemma}

Using the half-line results of Section~\ref{s3} we now obtain the second principal result of this paper:

\begin{theorem} \lb{t4.2}
Assume Hypotheses \ref{h2.2}, \ref{h4.1}, $\la_0, \la_1 \in \bbR \bs \si(T)$, and $(\la_0,\la_1)\cap\si_\ess(T)=\emptyset$. Then,
\begin{equation}
N((\la_0,\la_1); T) = \sum_{x \in \bbR}
N\big(\Psi_+(\la_0,x,x_0)^* J \Psi_-(\la_1,x,x_0)\big).    \lb{4.5}
\end{equation}
\end{theorem}
\begin{proof}
Let $\mu=\la_1-\la_0$. Since by assumption $\si_\ess(T)\cap[\la_0,\la_1]=\emptyset$, the spectrum of $T$ in $(\la_0,\la_1)$ consists of at most finitely many discrete eigenvalues, hence by Lemma~\ref{l2.3},
\begin{align}
N((\la_0,\la_1); T) = \dim\bigl(\ran\bigl(P\big((-\infty, -\mu^{-1}); R(\la_1)\big)\big)\big) < \infty.    \lb{4.6}
\end{align}
Employing \eqref{4.4} and the min-max principle as in the proof of Lemma~\ref{l3.5}, one obtains for every $a\in\bbR$,
\begin{align}
\dim\bigl(\ran\bigl(P\big((-\infty,-\mu^{-1}); R_a(\la_1)\big)\big)\big) \leq
\dim\bigl(\ran\bigl(P\big((-\infty,-\mu^{-1}); R(\la_1)\big)\big)\big) < \infty. \lb{4.7}
\end{align}
Thus, $R_a(\la_1)$ has no essential spectrum in $(-\infty,-\mu^{-1})$ and hence by  Lemma~\ref{l2.3} and \eqref{4.6}, \eqref{4.7}, one has
\begin{align}
N((\la_0,\la_1); T_a) = \dim\bigl(\ran\bigl(P\big((-\infty, -\mu^{-1}); R_a(\la_1)\big)\big)\big) \leq N((\la_0,\la_1); T).  \lb{4.8}
\end{align}

Since $P_a$ strongly converges to $I_{L^2_W(\bbR)^r}$ in $L^2_W(\bbR)^r$ as $a \downarrow - \infty$, $P_aR(\la_1)P_a$ strongly converges to $R(\la_1)$ in $L^2_W(\bbR)^r$ as $a \downarrow - \infty$. Then, as in Theorem~\ref{t3.6}, one obtains using \eqref{4.4},
\begin{align}
N((\la_0,\la_1); T) &= \dim\bigl(\ran\bigl(P\big((-\infty, -\mu^{-1}); R(\la_1)\big)\big)\big)   \no \\
& \leq \liminf_{a \downarrow -\infty} \,
\dim\bigl(\ran\bigl(P\big((-\infty,-\mu^{-1}); P_aR(\la_1)P_a\big)\big)\big) \no \\
& = \liminf_{a \downarrow -\infty} \,
\dim\bigl(\ran\bigl(P\big((-\infty,-\mu^{-1}); R_a(\la_1)\big)\big)\big).  \lb{4.9}
\end{align}
Combining \eqref{4.9} with \eqref{4.8} implies
\begin{align}
N((\la_0,\la_1); T) = \lim_{a \downarrow - \infty} \, N((\la_0,\la_1); T_a).
\end{align}
Applying Theorem \ref{t3.10} and noting that, by construction in \eqref{4.1}, $\Psi_-(\la_1,\dott,x_0)$ is the left Weyl--Titchmarsh solution for all $T_a$ then yields \eqref{4.5}.
\end{proof}

We emphasize again that the interval $(\la_0,\la_1)$ can lie in any essential spectral gap of $T$, not just below its essential spectrum as in standard approaches to oscillation theory. Again we note that to the best of our knowledge, even the special scalar case $m=1$ appears to be new for general Hamiltonian systems.

The analog of Remark \ref{r3.11} applies of course in the current full-line situation.

\smallskip

\noindent {\bf Acknowledgments.}
We are indebted to Selim Sukhtaiev for numerous discussions on this topic and
to the anonymous referee for a thorough reading of our manuscript and for his most
valuable comments. M.Z. gratefully acknowledges the kind invitation and hospitality
of the Mathematics Department of the University of Missouri during the spring semester 
of 2016, where much of this work was completed.


\end{document}